\documentclass{amsart}
\title{Integration with filters}
\author{Emanuele Bottazzi}
\author{Monroe Eskew}

\thanks{The second author wishes to acknowledge the support of the Austrian Science Fund (FWF) through Research Projects P28420 and Y1012.  He is also grateful to Hazel Brickhill, Leon Horsten, Ashutosh Kumar, and Benjamin Miller for some fruitful discussions.}

\date{}

\usepackage{color} 
\usepackage{amssymb} 
\usepackage{amsmath} 
\usepackage[utf8]{inputenc} 
\usepackage{enumerate}
\usepackage{amsthm}
\usepackage{cite}
\usepackage{enumerate}
\usepackage{esint}
\usepackage{comment} 

\newtheorem{theorem}{Theorem}
\newtheorem{lemma}[theorem]{Lemma}
\newtheorem{proposition}[theorem]{Proposition}
\newtheorem{corollary}[theorem]{Corollary}

\newtheorem{claim}[theorem]{Claim}
\newtheorem*{definition}{Definition}
\newtheorem{remark}[theorem]{Remark}

\newtheorem{fact}[theorem]{Fact}

\DeclareMathOperator{\dom}{dom}

\DeclareMathOperator{\len}{len}

\DeclareMathOperator{\st}{st}
\DeclareMathOperator{\dep}{dep}
\DeclareMathOperator{\pow}{Pow}
\DeclareMathOperator{\vol}{vol}
\newcommand{\p}{\mathcal{P}}
\newcommand{\la}{\langle}
\newcommand{\ra}{\rangle}
\newcommand{\R}{\mathbb{R}}
\newcommand{\Q}{\mathbb{Q}}
\newcommand{\Z}{\mathbb{Z}}
\newcommand{\N}{\mathbb{N}}
\newcommand{\civita}{\mathcal{R}}
\newcommand{\F}{k}
\newcommand{\sh}{\st}
\newcommand{\shsup}[1]{\st^+#1}
\newcommand{\shinf}[1]{\st^-#1}
\newcommand{\seq}[1]{[#1]_F}
\newcommand{\e}{\varepsilon}
\DeclareMathOperator{\fun}{Fun}
\DeclareMathOperator{\E}{E}

\begin{document}
\maketitle

\begin{abstract}
We introduce a notion of integration defined from filters over families of finite sets.
This procedure corresponds to determining the average value of functions whose range lies in any algebraic structure in which finite averages make sense.
The average values so determined lie in a proper extension of the range of the original functions.
The most relevant scenario involves algebraic structures that extend the field of rational numbers; hence, it is possible to associate to the filter integral an upper and lower standard part. These numbers can be interpreted as upper and lower bounds on the average value of the function that one expects to observe empirically.
We discuss the main properties of the filter integral and we show that it is expressive enough to represent every real integral.
As an application, we define a geometric measure on an infinite-dimensional vector space that overcomes some of the known limitations valid for real-valued measures. 
We also discuss how the filter integral can be applied to the problem of non-Archimedean integration, and we develop the iteration theory for these integrals.
\end{abstract}

In a magazine article \cite{quanta} on problems and progress in quantum field theory, Wood writes of Feynman path integrals, ``No known mathematical procedure can meaningfully average an infinite number of objects covering an infinite expanse of space in general. The path integral is more of a physics philosophy than an exact mathematical recipe.''  Though nonstandard analysis arguably provides a general method for averaging arbitrary collections of objects, perhaps its nonconstructive character keeps it from meriting the description ``exact mathematical recipe.''  In this manuscript, we present a constructive approach using integrals defined from filters over families of finite sets.  This is inspired by the ``non-Archimedean probability'' theory of Benci et al. \cite{MR3046984}
(that in turn drew inspiration from the early results of nonstandard measure theory before the development of Loeb measures, such as \cite{MR315082,10.2307/2041916}).


An advantage of our approach to integration over the classical one is its generality, as it allows us to determine the average value of functions whose range lies in any algebraic structure in which finite averages make sense.  A potential drawback is that the average values so determined typically lie in a proper extension of the algebraic structure with which we start.  In the case of real-valued functions, this means a partially ordered ring with infinite and infinitesimal elements.  However, this can also be seen as an advantage in that it allows for a more fine-grained quantification of the sizes of sets and the behavior of functions.  For example, different nonempty sets may be assigned different nonzero infinitesimal sizes, with the relation between these sizes corresponding to the limiting behavior of finite samples.
The empirical meaning of a series of relations like:
$$m(A_i) \ll m(A_{i+1}); \, m(B_i) \approx r_i m(A_i),$$
for $i \in \N$ and positive reals $r_i$, can be cashed out by saying that for all $i,n\in \mathbb N$, a generic finite sample of points $z$ will have 
$$\frac{|z \cap A_i|}{|z \cap A_{i+1}|} + \left| \frac{|z \cap B_i|}{|z \cap A_i|} - r_i \right| < \frac{1}{n}$$
Classical measures would flatten the description to just give all of these sets measure zero, erasing the information about such statistical phenomena.  

If we use filters that are maximal with respect to inclusion, i.e.\ ultrafilters, then the range of values for our integrals of ordered-field-valued functions will also be an ordered field.
If the functions are real-valued, the use of ultrafilters leads to the hyperfinite counting measures of nonstandard analysis. These measures, together with the Loeb measure construction \cite{loeb1,loeb2,loeb3}, have become the main tool of nonstandard measure theory and can be applied to the study of a variety of mathematical objects. Some examples include generalized functions \cite{cutland1986infinitesimal,BOTTAZZI2019429}, stochastic processes \cite{anderson1976non,keisler1984infinitesimal,perkins1981global,duanmu2021ergodicity}, statistical decision theory \cite{DR}, and mathematical economics \cite{duffie2018dynamic,khan1974some,sun1996hyperfinite,sun1999complete,khan1999non,khan2003exact,duffie2007existence,anderson2008equilibrium,sun2016independent}.

This convenient feature fails for non-maximal filters.  However, we can still do a lot while keeping to a definable setting and avoiding much use of the Axiom of Choice (AC).  In this way, our work has similarities with that of \cite{Henle1,Henle2,Laugwitz1,Laugwitz2}.
However, there are a few places where AC is invoked.  In Theorem \ref{ext}, Proposition \ref{ext_measurable}, and in \S\ref{geomsec}, we rely on background facts from classical analysis that depend on the axiom of countable choice (CC) or the axiom of dependent choices (DC).  In \S\ref{dimsec}, the full AC is used for a transfinite induction.  In \S\ref{sec decomposition}, the Hahn Embedding Theorem makes an appearance.  In \S\ref{transfinite}, AC appears in the form of assuming certain products are nonempty and in the use of Tychonoff's Theorem.

The structure of the manuscript is as follows. In \S\ref{prelim}, we introduce the basic facts and definitions.  In \S\ref{reps}, we show that our integrals can be used to represent many classical integrals.  This representation has the advantage that a complete real-valued measure and its filter representation are definable from one another, whereas a similar representation via hyperfinite counting measures does not carry such a correspondence.
In \S\ref{sec namg}, we discuss some applications of our integrals to geometry.  We construct a non-Archimedean measure on the direct limit of the $\R^n$ that overcomes some of the known limitations of real-valued measures on infinite-dimensional spaces, addresses the Borel-Kolmogorov paradox, and gives rise to a new notion of fractal dimension.  In \S\ref{sec nai}, we discuss the application of our technique to the problem of developing an appropriate notion of integral for non-Archimedean fields.
In \S\ref{products}, we introduce iterated integrals via product filters and discuss the interaction with the standard-part operation.  In \S\ref{transfinite}, we define transfinitely iterated integrals and discuss a few applications.

\section{Basic structures and operations}
\label{prelim}

The context in which our integrals can be defined is quite broad.  We need an infinite set $X$, a fine filter $F$ over $[X]^{<\omega}$, and a divisible Abelian group $G$.  Recall that a filter $F$ over $Z \subseteq \p(X)$ is \emph{fine} when for all $x \in X$, $\{ z \in Z : x \in z \} \in F$.  Since $F$ is a filter, this is equivalent to saying that for all finite $z_0 \subseteq X$, $\{ z \in Z : z_0 \subseteq z \} \in F$.
Recall that a group $G$ is \emph{divisible} when for all $a \in G$ and all positive $n \in \mathbb N$, there is $b \in G$ such that $nb := \underbrace{b + b + \dots + b}_{n \text{ times}} = a$.  By the group axioms, there is at most one such $b$, which we denote by $a/n$ or $n^{-1}a$.  

\subsection{Comparison rings}

Although our notion of integration will make sense for functions taking values in divisible Abelian groups, in the cases of interest, we want more than just a group structure.  Ideally, we would like to work with ordered fields, but our main operation will take us out of this category.  Thus we consider the following larger class of structures.   
Let us say that a structure is a \emph{comparison ring} if it is commutative ring with 1 and it carries a binary relation $<$ with the following properties:

\begin{enumerate}
\item\label{strictpo} $<$ is a strict partial order (i.e.\ transitive and irreflexive).
\item\label{addineq} For all $a,b,c$, if $a<b$, then $a+c<b+c$.
\item\label{multineq} For all $a,b$, if $a,b>0$, then $ab>0$.
\item\label{inversesquare} For all $a$, $a$ has a multiplicative inverse if and only if $a^2>0$.
\end{enumerate}

Let us list some elementary properties of comparison rings that will come in handy:

\begin{proposition}
\label{Karith}
Suppose $K$ is a comparison ring and $a,b,c,d \in K$.
\begin{enumerate}
\item\label{zero<one} $0<1$.
\item\label{+inv} If $a > 0$, then $a$ is invertible and $a^{-1}> 0$.
\item\label{-inv} If $a<0$, then $a$ is invertible and $a^{-1} = -(-a)^{-1}<0$.
\item\label{addineq2} If $a<b$ and $c<d$, then $a+c < b+d$.
\item\label{+scaling} If $a<b$ and $0<c$, then $ac < bc$.
\item\label{recip} $0<a<b$ if and only if $0<b^{-1}<a^{-1}$.
\item\label{Q} The ordered field $\Q$ of rational numbers is a substructure of $K$.
\end{enumerate}
\end{proposition}

\begin{proof}
 \hspace{1mm}
\begin{enumerate}
\item $1^{-1} = 1$ so by axiom (\ref{inversesquare}), $1 = 1^2 >0$.
\item If $a>0$, then by axiom (\ref{multineq}), $a^2>0$, so $a$ is invertible.  Then $a^{-1}$ is also invertible, so $(a^{-1})^2 > 0$.  Thus $a(a^{-1})^2 = a^{-1} > 0$.
\item If $a<0$, then axiom (\ref{addineq}) implies $-a > 0$, and $(-a)^2 = (-1)^2 a^2 = a^2$, so $a$ is invertible.  Further, $(-a^{-1})(-a) = (-1)^2 aa^{-1} = 1$, so $(-a)^{-1} = -a^{-1}$.  Since $(-a)^{-1} > 0$, $a^{-1} = -(-a)^{-1}<0$.
\item Applying axiom (\ref{addineq}), we have $a+c<b+c<b+d$.
\item Note that axiom (\ref{addineq}) implies $a<b$ iff $b-a > 0$.  By axiom (\ref{multineq}), $bc-ac > 0$.
\item Apply claims (\ref{+inv}) and (\ref{+scaling}) and multiply the inequalities by $a^{-1}b^{-1}$.
\item First we claim that the natural numbers appear in $K$ under the standard ordering (with $n$ represented in $K$ as $\underbrace{1 + 1 + \dots + 1}_{n \text{ times}}$). This follows by an induction using claim (\ref{zero<one}) and axioms (\ref{strictpo}) and (\ref{addineq}).  Next, for inequalities $-n<m$ among integers where $n>0$, use the inequality established previously for $0<m+n$, and then add $\underbrace{-1 + -1 + \dots + -1}_{n \text{ times}}$ to both sides and apply axiom (\ref{addineq}).  Next note that by claims (\ref{+inv}) and (\ref{-inv}), all nonzero integers have a multiplicative inverse in $K$.  Finally, let us verify that the ordering on the rationals in $K$ agrees with the standard one.  For rational numbers $p,q$, represent them as $p=ad^{-1}$, $q = bd^{-1}$, where $a,b,d$ are integers and $d>0$. Then $\Q  \models p<q$ iff $\Z \models a<b$.  Since the ordering of the integers in $K$ agrees with the ordering of $\Z$, $K \models a<b$ iff $\Z \models a<b$.  Multiplying both sides by $d^{-1}$ and applying claims (\ref{+inv}) and (\ref{+scaling}) yields $\Q \models p<q$ iff $K \models ad^{-1} < bd^{-1}$.
\qedhere
\end{enumerate}
\end{proof}


Note that a comparison ring $K$ is a divisible Abelian group, since by item (\ref{Q}), $n^{-1}$ exists in $K$ for each positive integer $n$.  For any $a \in K$, $n(n^{-1}a) = a$.

Let us introduce some terminology and notation.  Let $K$ be a comparison ring, and let $a,b \in K$.
\begin{itemize}
\item We say $a$ is \emph{finite} when $-n< a < n$ for some $n \in \mathbb N$, and \emph{infinite} when it is not finite.
Note that the set of finite elements forms a subring.
\item If $b>0$ and $-b<na<b$ for all $n\in\Z$, then we write $a \ll b$.  Note that 
the set $\{ a \in K : a \ll b \}$ is closed under addition and under multiplication by finite elements.
\item We say $a$ is \emph{infinitesimal} when $a \ll 1$.
\item We say $a \sim b$ when $a-b$ is infinitesimal.
\item We say $a \approx b$ when $b$ is invertible and $ab^{-1} \sim 1$.  Note that this implies $a$ is also invertible, because $1/2 < (ab^{-1})^2$ and so $0<b^2/2<a^2$.  Thus also $ba^{-1} \sim 1$.
\item We say that $a,b>0$ are \emph{Archimedean-equivalent} if there are $n,m \in \N$ such that $a<nb$ and $b<ma$.
\end{itemize}



\subsection{The reduced power construction}
We recall briefly the properties of the reduced power construction relevant for the development of the filter integral.
The interested reader can find a more general presentation with all the proofs we have omitted in Section V.2 of \cite{universal}.
The approach to infinite and infinitesimal numbers of Laugwitz \cite{Laugwitz1,Laugwitz2}, recently popularized by Henle \cite{Henle1,Henle2}, is also based on a similar reduced power construction of $\R$ with a different index set.

Suppose $K$ is a comparison ring, $Z$ is a set, and $F$ is a filter over $Z$.  Consider the ring $\fun(Z,K)$ of functions $f: Z \to K$.
Define an equivalence relation $\equiv_F$ on $\fun(Z, K)$ by putting $f \equiv_F g$ if and only if the set $\{ z : f(z) = g(z) \} \in F$.
We will denote by $[f]_F$ the equivalence class of $f$ in the quotient $\fun(Z, K)/\equiv_F$, which we will write as $\pow(K,F)$.
The 0 and 1 of $\pow(K,F)$ are interpreted as the equivalence classes of the constant functions with value 0 or 1 respectively in $K$.
Then the operations and the order relation on the quotient are defined pointwise modulo the filter:
\begin{itemize}
	\item $[f]_F+[g]_F=[h]_F$ iff $\{ z : f(z)+g(z) = h(z) \} \in F$;
	\item $[f]_F\cdot [g]_F=[h]_F$ iff $\{ z : f(z) \cdot g(z) = h(z) \} \in F$;
	\item $[f]_F < [g]_F$ iff $\{ z : f(z) < g(z) \} \in F$.
\end{itemize}
The above definitions do not depend on the choice of the representatives.

We can identify every element $a \in K$ with the equivalence class of the constant function $f_a(z) = a$ for every $z \in Z$, so we can identify $K$ with the set $\{ [f_a]_F : a \in K \}\subseteq \pow(K,F)$.
This identification induces a natural embedding $K \hookrightarrow \pow(K,F)$ (see e.g.\ Lemma 2.10 of Chapter V of \cite{universal}).
We will sometimes write $\seq{a}$, or even just $a$, instead of $\seq{f_a}$.

If $K$ is an ordered field, then usually $\pow(K,F)$ is not an ordered field, because if $F$ is not maximal, we lose the existence of multiplicative inverses for all nonzero elements and the totality of the ordering.
Suppose $X$ is an infinite set, $F$ is a fine filter on $[X]^{<\omega}$, and $K$ is a comparison ring.
Let $f : [X]^{<\omega} \to K$ be defined as
\begin{equation}\label{eqn zerodivisor}
f(z) =
\left\{
\begin{array}{ll}
1 & \text{if } |z| \text{ is even,}\\
0 & \text{if } |z| \text{ is odd.}
\end{array}
\right.
\end{equation}
Then $\seq{f} = 1$ if and only if $\{z \in [X]^{<\omega} : |z| \text{ is even}\} \in F$, $\seq{f} = 0$ if and only if $\{z \in [X]^{<\omega} : |z| \text{ is odd}\} \in F$. If neither set is in $F$, $\seq{f}\ne 0$ and $\seq{f} \ne 1$.
If $F$ is the minimal fine filter, the latter case is true.
Since $\seq{f}\ne 1$ and $\seq{f}\ne 0$ but  $\seq{f}(1-\seq{f}) = 0$, $\seq{f}$ is a zero-divisor.
Moreover, $\seq f$ is order-incomparable with $1-\seq f$ , $- \seq f$,
and with every $a \in K$ such that $0<a<1$.

However, it is easy to check that being a comparison ring is preserved:
\begin{lemma}
\label{comparisonpower}
If $K$ is a comparison ring and $F$ is a filter over a set $Z$, then $\pow(K,F)$ is also a comparison ring.
\end{lemma}

\begin{proof}[Proof (sketch)]
The verification of each axiom is easy, so let us just check the axiom on the existence of inverses as an example.
If $\seq f \cdot \seq f > \seq 0$, then for some $A \in F$, $f(z)^2 > 0$ for all $z \in A$.  Thus $f(z)^{-1}$ exists for all $z \in A$, and if $g(z) = f(z)^{-1}$ on $A$ and otherwise $g(z) = 0$, then $\seq f \cdot \seq g = \seq 1$.  Conversely, if $[f]_F$ has a multiplicative inverse $[g]_F$, then there is $A \in F$ such that $f(z)g(z) = 1$ for all $z \in A$.  Thus $f(z)^2 > 0$ for all $z \in A$, and so $\seq f^2 > \seq 0$.
\end{proof}

If $F$ is a fine filter, then $\pow(K,F)$ also contains infinite elements.  To see that this is the case, define $f(z) = |z|$ for every $z \in [X]^{<\omega}$.
Then for all $n \in \N$, $n < \seq{f}$, since $n < f(z)$ for large enough $z$.

An additional property of the order that will be relevant for our approach to integration is the following:  If $a$ is a positive infinitesimal in $K$, then $\seq{a} \ll \seq{f}$ for every positive $f : Z \to \Q$.
This is a consequence of the inequality $a \ll f(z)$ for every $z \in Z$.

\subsection{The standard part in the reduced power construction}
In the reduced power of a comparison ring $K$, it is in general false that every finite element is infinitesimally close to a real number. An example is provided by $\seq{f}$ with $f$ defined by equation \eqref{eqn zerodivisor}: if $F$ does not decide the equalities $\seq{f}=0$ and $\seq{f}=1$, $\seq{f}$ is finite but it is not infinitesimally close to any real number.
Thus, in general it is not possible to define a standard part for an element of $\pow(K,F)$.

However it is possible to define a superlinear ``lower standard part'' and a sublinear ``upper standard part''.

\begin{definition}\label{def upper and lower standard part}
	For a comparison ring $K$
	, we define the upper standard part of $a \in K$ as
	$$
	\shsup{a} =
	\inf\{ q \in \Q : a < q \}
	$$
	and the lower standard part of $a$ as
	$$
	\shinf{a} =
	\sup\{q \in \Q : a > q\}.
	$$
	We follow the convention that $\inf \emptyset = \sup \Q = \infty$ and $\sup \emptyset = \inf \Q = \ -\infty$.
	
	We say that $a\in K$ has a standard part if the upper standard part and the lower standard part are equal. In this case, we define $\sh{a} = \shsup{a} = \shinf{a}$.
\end{definition}

\begin{lemma}\label{lemma properties of shsup, shinf}
	For every finite $a,b \in K$,
	\begin{itemize}
		\item $\shsup{a} \geq \shinf{a}$;
		\item $\shsup{a} = - \shinf{(-a)}$;
		\item $\shsup{(a+b)} \leq \shsup{a} + \shsup{b}$;
		\item $\shinf{(a+b)} \geq \shinf{a} + \shinf{b}$;
		\item If $a,b \geq 0$, then $\shinf{a} \cdot \shinf{b} \leq \shinf{(a \cdot b)} \leq \shsup{(a \cdot b)} \leq \shsup{a} \cdot \shsup{b}$.
	\end{itemize}
\end{lemma}
\begin{proof}
	These inequalities follow from and the properties of $\inf$ and $\sup$ together with Proposition \ref{Karith}.
\end{proof}

\begin{lemma}
\label{finitestd}
Suppose $K$ is a comparison ring extending $\R$.  An element $a \in K$ has a finite standard part if and only if there is $r \in \R$ such that $a \sim r$.
\end{lemma}

\begin{proof}
Suppose first that $a - r$ is infinitesimal, where $r \in \R$.  Let $q < r$ be rational, and let $n \in \N$ be such that $r-q>1/n$.  Since $-1/n < a-r < 1/n$, we have $a-q = (a-r) + (r-q) > 0$, so $a>q$.  Similarly, $a<p$ for any rational $p>r$.  Hence, $\st(a) = r$.

For the other direction, suppose $\st(a) = r \in \R$.  Let $n \in \N$ be arbitrary.  Let $q_0,q_1$ be rational numbers such that $q_0<a<q_1$ and $q_1-q_0<1/n$.  We have that $a-r < q_1 - r<1/n$ and $r-a < r-q_0<1/n$.  Thus $-1/n < a-r < 1/n$.  Since $n$ was arbitrary, $a-r$ is infinitesimal.
\end{proof}

In general it is false that every finite invertible element of a comparison ring has a standard part. 
For instance, let $f$ be defined as in \eqref{eqn zerodivisor} and let $F$ be a filter that decides neither $\seq{f}=0$ nor $\seq{f}=1$.  Then $\seq{f}+1$ is invertible in $\pow(\Q,F)$ and its inverse is the equivalence class of the function
$$
g(z) =
\left\{
\begin{array}{ll}
\frac{1}{2} & \text{if } |z| \text{ is even,}\\
1 & \text{if } |z| \text{ is odd.}
\end{array}
\right.
$$
However $\shsup(\seq{f} +1) = 2 \ne 1 = \shinf(\seq{f} +1)$.

If $F$ is an ultrafilter and $K$ is an ordered ring, then it is well-known that every finite $a \in \pow(K,F)$ has a standard part.

\subsection{The filter integral}

\begin{definition}\label{def filter integral}
	Let $G$ be a divisible Abelian group and $F$ a fine filter over $[X]^{<\omega}$.
	We define an operator that assigns to functions $f : X \to G$ a value in $\pow(G,F)$. 
	$$\int f \, dF := \left[ z \mapsto \sum_{x \in z} f(x)/|z| \right]_F$$
\end{definition}

The values $\sum_{x \in z} f(x)/|z|$ give an approximation to the integral $\int f \, dF$ by looking at the average behavior of $f$ on finite snapshots of $X$.  They approximate it in the sense that we obtain $\int f \, dF$ by letting $z$ ``converge to $X$'' via $F$.

We have that for any $c \in G$, $\int c \, dF = \seq c$.  Furthermore, for any functions $f,g : X \to G$, $\int (f +g) \, dF = \int f \, dF + \int g \, dF$.  This is because
\begin{align*}
\int (f +g) \, dF	&= \left[z \mapsto  |z|^{-1} \sum_{x \in z} \left( f(x)+g(x) \right) \right]_F \\
			&= \left[z \mapsto  \sum_{x \in z} f(x)/|z| + \sum_{x \in z} g(x)/|z| \right]_F \\
			&= \left[z \mapsto  \sum_{x \in z} f(x)/|z|\right]_F  + \left[z \mapsto \sum_{x \in z} g(x)/|z| \right]_F \\
			&=  \int f \, dF + \int g \, dF.
\end{align*}
Moreover, when $G$ has a ring structure, the integral is a linear operator.

Suppose $K$ is a comparison ring.  In general the non-strict inequality is not very well-behaved in reduced powers of $K$.  For a filter $F$ on $Z$, it may be the case that for some functions $f,g : Z \to K$, we have $f(z) \leq g(z)$ for all $z \in Z$, but it is not the case that either $f(z) < g(z)$ on a set in $F$ or $f(z) = g(z)$ on a set in $F$.  However, integrals via fine filters behave better.  Suppose $F$ is a fine filter on $[X]^{<\omega}$, and $f,g : X \to K$ are functions such that $f(x) \leq g(x)$ for all $x \in X$.  Then either $f = g$, or there is an $x_0 \in X$ such that $f(x_0) < g(x_0)$.  In the latter case, for all $z \in [X]^{<\omega}$ with $x_0 \in z$, we have $\sum_{x \in z} f(x) < \sum_{x \in z} g(x)$, and thus $\int f \, dF < \int g \, dF$.  Thus we can say that if $f \leq g$, then $\int f \, dF \leq \int g \, dF$.


\begin{lemma}
\label{finiteslice}
Suppose $F$ is a fine filter over $[X]^{<\omega}$, $K$ is a comparison ring, and $f,g : X \to K$ are such that $f(x) \sim g(x)$ for all $x \in X$. Then $\int f \, dF \sim \int g \, dF$.
\end{lemma}

\begin{proof}
Let $\e(x) = f(x) - g(x)$.  
For all $z \in [X]^{<\omega}$
and $n \in \mathbb N$,
$$-1/n< |z|^{-1}\sum_{x \in z} \varepsilon(x) <1/n.$$
Therefore, $\int (f - g) \, dF = \int f \, dF - \int g \, dF$ is infinitesimal.
\end{proof}

\subsection{Standard integrals}
\label{sec:std}

\begin{definition}\label{standard integrals}
	If $K$ is a comparison ring and $F$ is a fine filter over $[X]^{<\omega}$, then for every $f : X \to K$ we define
	\begin{itemize}
		\item the \emph{upper integral} of $f$ as $\sqint^+ f \, dF := \shsup(\int f \, dF)$;
		\item the \emph{lower integral} of $f$ as $\sqint^- f \, dF := \shinf(\int f \, dF)$;
		\item the \emph{standard integral} of $f$ as $\sqint f \, dF := \st(\int f \, dF)$, if this is well-defined.
	\end{itemize}
\end{definition}

If $F$ is an ultrafilter, then every function has a standard $F$-integral.

By Lemma \ref{lemma properties of shsup, shinf}, for all functions $f,g : X \to K$ with finite integral and for all $r \in \Q$,
\begin{itemize}
	\item $\sqint^+ (f+g) \, dF \leq \sqint^+ f \, dF + \sqint^+ g \, dF$.
	\item $\sqint^- (f+g) \, dF \geq \sqint^- f \, dF + \sqint^- g \, dU$.
	\item $\sqint (f+g) \, dF = \sqint f \, dU + \sqint g \, dF$, if both terms on the righthand side are defined.
	\item $\sqint^{\pm} rf \, dF = r \sqint^{\pm} f \, dF$ if $r \geq 0$, and $\sqint^{\pm} rf \, dF = r \sqint^{\mp} f \, dF$ if $r <0$.
\end{itemize}
Thanks to the above properties, the set $\{ f \in \fun(X,K) : f \text{ has a standard integral} \}$ is a vector space over $\Q$. If $K \supseteq \R$, then in the above assertions, $\Q$ can be replaced with $\R$.
Moreover, if $F'$ is a filter extending $F$, then for all $f : X \to K$,
$$\sqint^- f \, dF \leq \sqint^- f \, dF' \leq \sqint^+ f \, dF' \leq \sqint^+ f \, dF.$$
This holds because for all rational numbers $q_0,q_1$, the relation ``$q_0 < \int f \, dF < q_1$'' means that for some $A \in F$, $q_0 < |z|^{-1}\sum_{x \in z} f(x) < q_1$ for all $z \in A$, and this $A$ would be in $F'$ as well.
It follows that if $\sqint f \, dF$ exists, then so does $\sqint f \, dF'$, and it is the same number.

Unfortunately, the collection of functions possessing a standard $F$-integral is in general not a ring.  For example, let $F$ the filter generated by the sets $A_n = \{ z \in [\omega]^{<\omega} : z$ is an initial segment of length $\geq n \}$.  One may construct two sets $A,B \subseteq \omega$ such that $\sqint \chi_A \, dF = \sqint \chi_B \, dF = 1/2$, but the function $\chi_A \chi_B = \chi_{A \cap B}$ does not have a standard integral because the density of the intersection oscillates between nearly half and nearly zero.  On the other hand, we will see in \S\ref{reps} that for many canonical filters, the class of functions possessing a standard integral is closed under multiplication and other operations.

One may interpret the upper and lower $F$-integrals of a function $f$ as upper and lower bounds on the average value of $f$ that one expects to observe empirically.  Similarly, one may interpret the upper and lower integrals of the characteristic function of a set as a confidence interval for the event described by the set.  The gap between these values can be reduced or even closed by encoding additional information, i.e.\ by considering the integrals induced by a filter $F' \supseteq F$.  This can be done by simply adding a single set to $F$ and closing under intersections and supersets.  Thus the filters can be readily updated to accommodate new data.

\subsection{Weighted integrals}
\label{weighted}

We would like to allow the possibility for some parts of our space to contribute to the approximation of the integral without having their contribution diminished as more points are added.  This will allow for point masses and for spaces with infinite volume.  Let $X$ be a set and let $\vec P = \{ P_i : i \in I \}$ be a partition of $X$.  Let $F$ be a fine filter over $[X]^{<\omega}$, and let $G$ be a divisible Abelian group.  For a function $f : X \to G$, we define 
$$\int f \, d(F,\vec P) = \left[ z \mapsto \sum_{i \in I} \sum_{x \in z \cap P_i}  f(x)/|z \cap P_i| \right]_F$$
Since each relevant $z$ is finite, each sum above involves only finitely many terms.  As before, $\int (f+g) \, d(F,\vec P) = \int f \, d(F,\vec P) + \int g \, d(F,\vec P)$.  

For each $i \in I$, let $\pi_i : [X]^{<\omega} \to [P_i]^{<\omega}$ be the map $z \mapsto z \cap P_i$.  For each $i$, $F$ canonically projects to a fine filter $F_i$ over $[P_i]^{<\omega}$ via the criterion $A \in F_i \Leftrightarrow \pi_i^{-1}[A] \in F$.  Each $\pow(G,F_i)$ canonically embeds into $\pow(G,F)$, via the map $[f]_{F_i} \mapsto [f \circ \pi_i]_F$.
If $\vec P$ is a finite partition $\{ P_i : i \leq n \}$, then for any $f : X \to G$, 
$\int f \, d(F,\vec P) = \sum_{i=0}^n \int (f \restriction P_i) \, dF_i,$
where we compute the sum of values from different reduced powers via the canonical embeddings.

\subsection{Probabilities and a countable example}

If $F$ is a fine filter over $[X]^{<\omega}$, then we define the \emph{$F$-probability} of a set $A \subseteq X$ as $\int \chi_A \, dF$.  This is also written as $\Pr_F(A)$.
The \emph{expected value} (according to $F$) of a function $f$ on $X$ is $\int f \, dF$.  This is written as $\E_F(f)$.  We drop the subscript for the filter when it is clear from context.

We define the \emph{conditional expectation} of a function $f$ on a nonempty set $A$, $\E(f|A)$, as $(\int f \cdot \chi_A \, dF) /\Pr_F(A)$.  Since the filter $F$ is assumed to be fine, $\Pr_F(A)$ is always an invertible element of the comparison ring $\pow(\Q,F)$ when $A$ is nonempty, so the conditional expectation is always well-defined.
In contrast, it is well-known that, in the Kolmogorov model for probability, the problem of determining the conditional probability with respect to a set of null probability is not well posed \cite{Rao1988}. A typical approach that allows one to define $P(B|A)$ for sets satisfying $P(A) = 0$ consists in considering the limit $\lim_{n\to\infty} P(B|A_n)$ under the hypotheses that $\lim_{n\to\infty} A_n = A$ and $P(A_n) > 0$ for all $n\in\N$. However this limit depends on the choice of the sequence $\{A_n\}_{n\in\N}$. We discuss a more concrete example in Section \ref{sec bk}.

For a set $B \subseteq X$, we write $\Pr(B|A)$ for $\E(\chi_B | A)$, which is always a member of the comparison ring between $0$ and $1$.  $\E(f|A)$ can also be directly expressed in the reduced power as:
$$ \left[ z \mapsto \sum_{x \in z \cap A} f(x) / | z \cap A | \right]_F$$
These notions also make sense for weighted integrals.  Suppose $\vec P$ is a partition of $X$ and $A \subseteq X$ is nonempty.  Then the average or expected value $\E(f|A)$ of a function $f$ on $A$ is defined as $\int f \cdot \chi_A  \, d(F,\vec P) / \int \chi_A \, d(F,\vec P)$.  Thus it makes sense to compute conditional expectations using any nonempty condition, even those of infinitesimal or infinite measure.

As discussed in \cite{MR3046984}, this kind of notion allows for ``fair'' probability distributions on infinite sets (even countable sets), where the probability of any single point is the same nonzero value, contrary to the classical situation.  We would like to briefly discuss a similar class of examples that allows us to naturally model the notion of independent random variables using only hereditarily countable mathematical objects.  The classical treatment uses infinite products of measure spaces, which involves objects of size at least continuum \cite{MR3930614}.  

Fix a natural number $k\geq 2$.  Let $T_k$ be the complete $k$-ary tree of height $\omega$.  Our set $X$ consists of the nodes of $T_k$, i.e.\ the finite $k$-ary sequences.  For each $n<\omega$, let $T_k^n$ be the set of all $k$-ary sequences of length $\leq n$.  Let $Z \subseteq [X]^{<\omega}$ be the collection of all $T_k^n$.  Let $F$ be the smallest fine filter on $Z$, i.e.\ the one generated by the sets $\{ T_k^m : m \geq n \}$ for $n<\omega$.

For $n<\omega$ and $i<k$, let $A_n = \{ s \in T_k : \len(s) > n \}$ and let $B_{i,n} = \{ s \in A_n : s(n) = i \}$.  It is easy to see that for all $n$, $\st(\Pr(A_n)) = 1$, and $\Pr(B_{i,n} | A_n) = 1/k$.  Further, for distinct $n_1,\dots,n_{r}<\omega$ and any $i_1,\dots,i_{r}<k$, 
$$\st(\Pr(B_{i_1,n_1} \cap \dots \cap B_{i_{r},n_{r}} )) = 1/k^r.$$

If we want to model independent trials for which the probabilities can take on a wider range of values, we can consider the space $T_{\Q}$, the set of finite sequences of rational numbers between 0 and 1.  To define the appropriate filter, consider for each $n,m<\omega$, the subtree $T_{1/m}^n$ consisting of all finite sequences of length $\leq n$ such that each coordinate is of the form $k/m$, where $0\leq k \leq m$ is an integer.  Let $F$ be the smallest fine filter over the set of all $T_{1/n}^n$.

For $n<\omega$ and reals $0 \leq a \leq b \leq 1$, let 
$$B^n_{a,b} = \{ s \in T_{\Q} : \len(s) > n \wedge a \leq s(n) < b \}.$$
It is easy to see that for all distinct $n_1,\dots,n_r<\omega$ and all choices of intervals $[a_1,b_1),\dots,[a_r,b_r)$,
$$\st(\Pr(B^{n_1}_{a_1,b_1} \cap \dots \cap B^{n_r}_{a_r,b_r})) = (b_1-a_1)\cdots(b_r-a_r).$$
This is because, given any $\e>0$, if we take $m$ large enough, then the proportion of points in $\prod_{1\leq i \leq r} [0,1]_{\Q}$ with denominator $1/m$, that lie in the rectangle $\prod_{1\leq i \leq r} [a_i,b_i)$, is within $\e$ of the classical volume of this rectangle.

\section{Representations of classical integrals}
\label{reps}

In this section we show that the filter integral is general enough to represent every real-valued measure defined on an algebra $\mathcal A$ of subsets of $X$.
Using the ``hyperfinite'' approach as in \cite{MR315082}, we can obtain similar results involving ultrafilters.  However, we work here to define the filters directly from given measures.


The following lemma is a slight strengthening of one appearing in \cite{DBLP:journals/jla/BenciBN14}.

\begin{lemma}\label{combinatorial}
	Suppose $\mu$ is a finitely additive measure defined on an algebra $\mathcal A$ of subsets of an infinite set $X$, taking extended real values in $[0,\infty]$ and giving measure zero to all singletons.  Let $Y_1,\dots,Y_k \in \mathcal A$ have finite measure, let $x_1,\dots,x_\ell \in X$, and let $n \in \N$ be positive. 
	There exists a finite $z\subseteq X$
	that satisfies the following properties:
	
	\begin{enumerate}
		\item
		$x_1,\ldots,x_\ell\in z$;
		\item $n\ell< |z|$;
		\item if $\mu(Y_1 \cup\dots\cup Y_k) > 0$, then $z \setminus \{x_1,\ldots,x_\ell\} \subseteq Y_1 \cup\dots\cup Y_k$;
		\item	for $1 \leq i,j \leq k$, if $\mu(Y_i)\not= 0$, then:
		$$\left|\frac{|z\cap Y_j|}{|z\cap Y_i|} -
		\frac{\mu(Y_j)}{\mu(Y_i)}\right|\ <\ \frac{1}{n}$$
	\end{enumerate}
\end{lemma}

\begin{proof}
Let $r = \mu(Y_1 \cup\dots\cup Y_k)$.  We may assume $r>0$, since otherwise the conclusion is trivial.
For $1 \leq i\leq k$, let $Y_i^0 = Y_i$ and $Y_i^1 = X \setminus Y_i$.  
Consider all Boolean combinations of the form
$Y_1^{i_1} \cap \dots \cap Y_k^{i_k}$,
where $i_j = 0$ for at least one value of $j$.  List all such combinations that have positive measure as $\{ B_i : 1 \leq i \leq N \}$.  Note that these are pairwise disjoint infinite sets, and $\sum_{i=1}^N \mu(B_i) = r$.

Let $s$ be the minimum positive value of $\mu(Y_i)$ for $1 \leq i \leq k$, and let $\e>0$ be smaller than $\min\{ 1/n, s/2, s^2/4rn \} $.
For $1\leq i<N$, let $q_i$ be a positive rational number such that 
$$\mu(B_i)/r -\e/2N^2 <q_i<\mu(B_i)/r.$$
Let $q_N = 1 - \sum_{i=1}^{N-1} q_i$, so that each $q_i$ is positive and $\sum_{i=1}^N q_i = 1$.  It follows that $0<q_N - \mu(B_N)/r < \e/2N$.
For $1 \leq j \leq k$, we have that $\mu(Y_j)= \sum_{B_i \subseteq Y_j} \mu(B_i)$.  Thus 
$| \mu(Y_j)/r - \sum_{B_i \subseteq Y_j} q_i | < \e/2$.

Now take a sufficiently large common denominator $M$ for the $q_i$ such that for $1 \leq i \leq N$, there is a natural number $p_i$ with $p_i/M = q_i$ and $\ell/p_i < \e/2$.
Then choose $z' \in [X]^{<\omega}$ such that:
\begin{enumerate}
\item $|z'| = M$.
\item $z' \subseteq \bigcup_{i=1}^N B_i$;
\item for $1 \leq i \leq N$, $|z' \cap B_i | = p_i$.
\end{enumerate}
Let $z = z' \cup \{ x_1,\dots,x_\ell\}$. For each $Y_j$, $1 \leq j \leq k$, let $\ell_j = |(z \setminus z') \cap Y_j|$.  Then:
$$\ \left| \frac{|z \cap Y_j|}{|z|} -\frac{|z' \cap Y_j|}{M} \right| = \left| \frac{|z' \cap Y_j| + \ell_j}{M+\ell_j} -\frac{|z' \cap Y_j|}{M} \right| = \frac{\ell_j(M-|z' \cap Y_j|)}{M(M+\ell_j)} \leq \frac{\ell}{M} < \frac{\e}{2}$$
Since $|z' \cap Y_j| = \sum_{B_i \subseteq Y_j} p_i$,
$$\left| \frac{|z \cap Y_j|}{|z|} - \frac{\mu(Y_j)}{r} \right| < \left| \sum_{B_i \subseteq Y_j} q_i - \mu(Y_j)/r \right| + \e/2 < \e.$$


Now suppose $1\leq i,j \leq k$ and $\mu(Y_i) > 0$.  If $\mu(Y_j) = 0$, then 
$$|z\cap Y_j|/|z\cap Y_i| \leq \ell/p_i < \e < 1/n,$$ 
so the desired conclusion holds.  If $\mu(Y_j) > 0$, then set $e_x = |z \cap Y_x| / |z|$ and $m_x = \mu(Y_x)/r$ for $x = i,j$.  We have:
\begin{align*}
\left|\frac{|z\cap Y_j|}{|z\cap Y_i|} - \frac{\mu(Y_j)}{\mu(Y_i)}\right|\ = \left| \frac{e_j}{e_i} - \frac{m_j}{m_i} \right| = \left| \frac{e_j m_i - e_i m_j}{e_i m_i} \right| \leq \left| \frac{e_j m_i - e_i m_j}{s^2/2} \right| \\
 \leq \left| \frac{ (m_j + \e)m_i - (m_i - \e) m_j}{s^2/2} \right| = \frac{\e(m_i + m_j)}{s^2/2} \leq \frac{4r\e}{s^2} < \frac{1}{n}
\end{align*}
\end{proof}

\begin{theorem}
\label{probext}
Suppose $\mu$ is a finitely additive real-valued probability measure defined on an algebra $\mathcal A$ of subsets of $X$, giving measure zero to all singletons.  Then there is a definable filter $F_\mu$ over $[X]^{<\omega}$, which is the smallest fine filter $F$ with the property that for any bounded $\mu$-measurable function $f : X \to \mathbb R$,
$$\int f \, d\mu = \sqint f \, dF.$$
\end{theorem}

\begin{proof}
For $x \in X$, let $A_x = \{ z \in [X]^{<\omega} : x \in z \}$, and for a set $Y \in \mathcal A$ and $n \in \N$, let $A_{Y,n} = \{ z : \left||Y \cap z|/|z| - \mu(Y) \right|< 1/n \}$.  By Lemma \ref{combinatorial}, the collection of all $A_x$ and $A_{Y,n}$ for $x \in X$, $Y \in \mathcal A$, and $n \in \N$, has the finite intersection property.  Let $F_\mu$ be the generated filter.

Suppose $F$ is any filter with the desired property.  Then for every $Y \in \mathcal A$, $\mu(Y) = \int \chi_Y \, d\mu = \sqint \chi_Y \, dF$.  This implies that for every $n \in \N$, the set of $z \in [X]^{<\omega}$ such that $||Y \cap z|/|z| - \mu(Y)| < 1/n$, is a member of $F$.  Thus $F_\mu$ is contained in any fine filter with the desired property.

Let $f$ be a bounded $\mu$-measurable function, and let $M \in \mathbb R$ be such that $|f|< M$.  For real numbers $a<b$, the set 
$$E_{a,b} := \{ x : a < f(x) \leq b \}$$
is in $\mathcal A$.  For a positive $n \in \N$, let $g_n$ be the function that takes value $Mi/n$ on $E_{Mi/n, M(i+1)/n}$ for $-n\leq i < n$, and let $h_n$ the function that takes value $M(i+1)/n$ on $E_{Mi/n, M(i+1)/n}$.  By the linearity of the integrals, for each $n$,
\begin{align*}
&\int g_n \, dF_\mu \leq \int f \, dF_\mu \leq \int h_n \, dF_\mu;\\
&\int g_n \, d\mu = \sqint g_n \, dF_\mu = \sum_{i=-n}^{n-1} \frac{Mi}{n} \mu(E_{Mi/r, M(i+1)/r}); \\
&\int h_n \, d\mu = \sqint h_n \, dF_\mu = \sum_{i=-n}^{n-1} \frac{M(i+1)}{n} \mu(E_{Mi/r, M(i+1)/r}); \\
&\sqint h_n \, dF_\mu - \sqint g_n \, dF_\mu = M/n.
\end{align*}
It follows that $\lim_{n \to \infty} \int g_n \, d\mu = \lim_{n \to \infty} \int h_n \, d\mu = \int f \, d\mu = \sqint f \, dF_\mu$.
\end{proof}

Suppose $\mu$ is a finitely additive real-valued probability measure defined on an algebra $\mathcal A \subseteq \p(X)$.  For $Y \subseteq X$, let $\mu^+(Y) = \inf\{ \mu(A) : A \in \mathcal A$ and $Y \subseteq A \}$, and let $\mu^-(Y) = \sup \{ \mu(A) : A \in \mathcal A$ and $Y \supseteq A \}$.  Say a set $Y$ is \emph{$\mu$-measurable} if $\mu^-(Y) = \mu^+(Y)$.  It is not hard to check that the collection of $\mu$-measurable sets forms an algebra $\bar{\mathcal A}$, and if we define $\bar\mu(Y) = \mu^-(Y)=\mu^+(Y)$ for $Y \in \bar{\mathcal A}$, then $\bar\mu$ is a finitely additive measure on $\bar{\mathcal A}$.


\begin{proposition}
\label{probext_measurable}
Suppose $\mu$ is a finitely additive probability measure defined on an algebra $\mathcal A \subseteq \p(X)$ that gives measure zero to all singletons.  Then for $Y \subseteq X$, $\chi_Y$ has a standard $F_\mu$-integral if and only if $Y$ is $\mu$-measurable.
\end{proposition}

\begin{proof}
If $Y \subseteq X$ is $\mu$-measurable, then for every $\e>0$, there are $A,B \in \mathcal A$ such that $A \subseteq Y \subseteq B$ and
$$\bar\mu(Y) - \e < \mu(A) = \sqint \chi_A \, dF_\mu \leq  \sqint \chi_B \, dF_\mu = \mu(B) < \bar\mu(Y)+\e$$
Since $\int \chi_A \, dF_\mu \leq  \int \chi_Y \, dF_\mu \leq \int \chi_B \, dF_\mu$, we have that $\sqint \chi_Y \, dF_\mu = \bar\mu(Y)$.

For the other direction, a result of \L o\'s and Marczewski  \cite{MR35327} shows that, if $Y \subseteq X$ and $\mu^-(Y) \leq r \leq \mu^+(Y)$, then we can define a measure $\nu$ on the algebra generated by $\mathcal A \cup \{ Y \}$ such that $\nu(Y) = r$ and $\nu \restriction \mathcal A = \mu$.  By Theorem \ref{probext}, we have that $F_\nu \supseteq F_\mu$, and $\sqint \chi_Y \, dF_\nu = \nu(Y)$.  Thus if $Y$ is not $\mu$-measurable, there are extensions of $\mu$ that give different values to $Y$.  Thus $\chi_Y$ cannot have a standard $F_\mu$-integral.
\end{proof}

\begin{theorem}
\label{ext}
Suppose $\mu$ is a countably additive, real-valued, $\sigma$-finite measure defined on a $\sigma$-algebra $\mathcal A$ of subsets of $X$.  Then there is a countable partition $\vec P$ of $X$, a fine filter $F$ over $[X]^{<\omega}$, definable from $\mu$ and $\vec P$, and a ``weight function'' $w : X \to \mathbb R$, constant on each $P_i$, such that $F$ is the smallest fine filter $G$ with the property that for any $\mu$-integrable function $f : X \to \mathbb R$,
$$\int f \, d\mu = \sqint fw \, d(G,\vec P).$$
Furthermore, if $\mu(X)<\infty$ then we can take $\vec P = \la P_i : i <\alpha \leq \omega \ra$ such that $P_0$ contains no point masses, and $P_i$ is a singleton for $0<i<\alpha$.
\end{theorem}
\begin{proof}

First note that by $\sigma$-finiteness, there can be only countably many point masses.
Let $X_0$ be the set of point masses, let $\{ P^0_i : i <\alpha \}$ partition $X_0$ into singletons, where $\alpha\leq\omega$, and let $w(x) = \mu(\{ x \})$ for $x \in X_0$.  By countable additivity, $\mu(Y) = \sum_{x \in Y} w(x)$ for all $Y \subseteq X_0$.
Let $X_1 = X \setminus X_0$.  If $\mu(X_1) = \infty$, let $\{ P^1_i : i \in \mathbb N \}$ be a partition of $X_1$ into sets of finite measure.   If $\mu(X_1) <\infty$, let $P^1_0= X_1$.  For $x \in P^1_i$, let $w(x) = \mu(P^1_i)$.

For $x \in X$, let $A_x = \{ z \in [X]^{<\omega} : x \in z \}$, and for an integrable function $f$ and $\e>0$, and let 
$$A_{f,\e} = \left\{ z : \left| \int f \, d\mu - \sum_{i,j} \sum_{x \in z \in P^i_j} \frac{f(x)w(x)}{|z \cap P^i_j |} \right| < \e \right\}$$
Let $F$ be generated by closing this collection of sets under intersection and supersets.  Clearly any filter satisfying the desired equations must contain all of these sets.  We must check that $F$ is a filter.  It will suffice to consider only nonnegative integrable functions $f$, since by breaking $f$ into the sum of its positive and negative parts and taking $\e$ small enough, we see that the same collection $F$ is generated.

Let $x_0,\dots,x_{m-1} \in X$, and let $f_0,\dots,f_{n-1}$ be $\mu$-integrable nonnegative functions.  Let $\varepsilon >0$ be given.  Using the countable additivity of $\mu$, we can find a large enough $N \in \mathbb N$ such that, if $A_k = \bigcup_{i<N} P^k_i$, then $x_0,\dots,x_{m-1} \in A_0 \cup A_1$, and for $i<n$ and $k<2$, 
$$\int_{X_k} f_i \, d\mu - \int_{A_k} f_i \, d\mu < \varepsilon/4.$$

For $r \in \mathbb R$, let $E_r = \{ x \in X : (\forall i < n) f_i(x) < r \}$.  Again using the countable additivity of $\mu$ (more specifically, the Monotone Convergence Theorem), we can find a large enough $M \in \mathbb R$ such that $x_0,\dots,x_{m-1} \in E_M \cap (A_0 \cup A_1)$, and for $i<n$ and $j<N$,
$$\left| \frac{\mu(P^1_j)}{\mu(P^1_j \cap E_M)} \int_{P^1_j \cap E_M} f_i \, d\mu - \int_{P^1_j} f_i \, d\mu \right| < \frac{\varepsilon}{4N}$$

For $i<n$ and $0 \leq a<b \leq M$, consider the set 
$E^{a,b}_i := \{ x : a < f_i(x) \leq b \}$.
By partitioning $[0,M]$ into small enough subintervals, we can apply Lemma \ref{combinatorial}
 to expand $X_1 \cap \{ x_0, \dots, x_m \}$ to a finite set $z' \subseteq A_1 \cap E_M$, so that for each $i < n$ and $j<N$,
$$\left| \frac{\int_{P^1_j \cap E_M} f_i \, d\mu}{\mu(P^1_j \cap E_M)} - \sum_{x \in z' \cap P^1_j} \frac{f_i(x)}{|z' \cap P^1_j|} \right| < \frac{\varepsilon}{4N\mu(P^1_j)}$$
Multiplying by $\mu(P^1_j)$ and combining with the previous inequality, we get that
$$\left| \int_{P^1_j} f_i \, d\mu - \sum_{x \in z' \cap P^1_j} \frac{f_i(x)w(x)}{|z' \cap P^1_j|} \right| < \frac{\varepsilon}{2N}$$
Let $z = A_0 \cup z'$.  Note that for each $i<n$,
$$\int f_i \, d\mu - \sum_{\substack{j<N \\ k < 2}} \sum_{x \in z \cap P^k_j} f_i(x)w(x)/|z \cap P^k_j | = \int_{X_0} f_i , d\mu - \sum_{x \in z \cap X_0} f_i(x)w(x)$$
$$+ \sum_{j<N} \left( \int_{P^1_j} f_i \, d\mu - \sum_{x \in z \cap P^1_j} f_i(x)w(x)/|z \cap P^1_j | \right) + \left( \int_{X_1} f_i \, d\mu -\int_{A_1} f_i \, d\mu \right)$$
The absolute value of this number is bounded by
$\varepsilon/4 + N(\varepsilon/2N) + \varepsilon/4 = \varepsilon.$
\end{proof}

Recall that a measure is \emph{complete} when all subsets of measure zero sets are measurable.  Every measure has a minimal extension to a complete measure with the same additivity.  Suppose $\mu$ is a probability measure on $X$.  For a bounded function $f : X \to \R$, let 
$$\int^- f \, d\mu = \sup\left\{ \int g \, d\mu : g \leq f \text{ and } g \text{ is measurable}\right\}$$
$$\int^+ f \, d\mu = \inf\left\{ \int g \, d\mu : g \geq f \text{ and } g \text{ is measurable}\right\}$$
When $\mu$ is countably additive, the Monotone Convergence Theorem implies that there are measurable functions $f_\ell,f_u$ such that $f_\ell \leq f \leq f_u$, and $\int f_\ell \, d\mu = \int^- f \, d\mu$, and $\int f_u \, d\mu = \int^+ f \, d\mu$.
The following is well-known:
\begin{fact}
Suppose $\mu$ is a countably additive complete probability measure on $X$, and $f : X \to \R$ is bounded.  The following are equivalent:
\begin{enumerate}
\item $\int^- f \, d\mu = \int^+ f \, d\mu$.
\item $\mu(\{ x : f_\ell(x) < f_u(x) \}) = 0$.
\item $f$ is $\mu$-measurable.
\end{enumerate}
\end{fact}

\begin{proposition}
\label{ext_measurable}
Suppose $\mu$ is a countably additive complete probability measure defined on a $\sigma$-algebra $\mathcal A \subseteq \p(X)$.  Let $f : X \to \R$ be bounded.  The following are equivalent:
\begin{enumerate}
\item $f$ is $\mu$-measurable.
\item $f$ has a standard $(F,\vec P)$-integral, where $\vec P$ is a partition according to Theorem \ref{ext} and $F$ is the canonical filter.
\end{enumerate}
\end{proposition}

\begin{proof}
The direction (1) $\Rightarrow$ (2) follows from Theorem \ref{ext}.  
For the other direction, assume for simplicity that $\mu$ has no point masses, so that we can ignore $\vec P$.
Let $\mu^*$ be the outer measure on $\p(X)$ induced by $\mu$.  Suppose $f : X \to \R$ is a bounded function that is not measurable.  By countable additivity, there is some $\e>0$ such that $\mu(\{ x : f_u(x) \geq f_\ell(x) + \e \})>0$.  Thus $\int f_\ell \, d\mu <\int f_u \, d\mu$.
  Now we claim that for all $\e > 0$,
$$\mu^*(\{ x : f(x) - f_\ell(x) < \e \}) = \mu^*(\{ x : f_u(x) - f(x) < \e \})  = 1.$$
Towards a contradiction, suppose that for some $\e,\delta>0$, 
$$\mu^*(\{ x : f(x) - f_\ell(x) < \e \}) = 1-\delta.$$  Let $E \in \mathcal A$ be such that $E \supseteq \{ x : f(x) - f_\ell(x) < \e \}$ and $\mu(E) = 1-\delta$.  Define:
$$g(x) = 
\begin{cases}
f_\ell(x) 		& \text{if } x \in E \\
f_\ell(x) + \e 	& \text{if } x \in X \setminus E
\end{cases}$$
Then $g$ is measurable, $g \leq f$, and $\int g \, d\mu -\int f_\ell \, d\mu = \e\delta > 0$.  Thus $\int g \, d\mu > \int^- f \, d\mu$, a contradiction.
We can show similarly that $\mu^*(\{ x : f_u(x) - f(x) < \e \})  = 1$.

It follows that for all $A \in \mathcal A$ of positive measure and all $\e>0$,
$$ \mu(A) = \mu^*(\{ x \in A : f(x) - f_\ell(x) < \e \} ) = \mu^*( \{ x \in A : f_u(x) - f(x) < \e \} ).$$ 
In particular, each set above is infinite.  Now, recalling the proofs of Lemma \ref{combinatorial} and Theorem \ref{ext}, we can use this to show that the following collection generates a filter $F_\ell$ over $[X]^{<\omega}$:
\begin{itemize}
\item $A_x$ for $x \in X$;
\item $A_{h,\e}$ for $\mu$-integrable $h : X \to \R$ and $\e>0$;
\item $\{ z : | \sum_{x \in z} f(x)/|z| - \int f_\ell \, d\mu | < \e \}$ for $\e>0$.
\end{itemize}
We have that $\sqint f \, dF_\ell = \int^- f \, d\mu$.  There is an analogous filter $F_u$ such that $\sqint f \, dF_u = \int^+ f \, d\mu$. If $F$ is the minimal filter given by Theorem \ref{ext}, then $F_\ell,F_u \supseteq F$.  Thus the function $f$ does not have a standard $F$-integral.
\end{proof}

Very similar conclusions can be drawn about functions that are bounded above and below by integrable functions.



\section{Non-Archimedean measures and geometry}
\label{sec namg}

\subsection{A geometric measure on $\R^{<\omega}$}
\label{geomsec}

A well-known no-go result in 
functional analysis states that there is no analogue of Lebesgue measure on infinite-dimensional separable Banach spaces such that:
\begin{itemize}
	\item every Borel set is measurable;
	\item the measure is translation-invariant;
	\item every point has a neighborhood with finite measure.
\end{itemize}
In the study of measures over infinite-dimensional spaces it is therefore usual to renounce $\sigma$-finiteness \cite{10.2307/2048779}.  This result is based on the following more general fact:  If $X$ is an infinite-dimensional normed vector space over the reals, then every open ball contains an infinite collection of pairwise-disjoint open balls of equal radius (in fact only 1/4 the radius of the original ball).  
Thus there cannot exist even a finitely additive translation-invariant measure on an infinite-dimensional normed real vector space that gives every open ball of finite radius a positive real measure.

We give a construction here of a non-Archimedean measure on a rather concrete space that contrasts with these impossibility results.  It will be translation-invariant (in a reasonable sense) on a wide class of sets that includes open balls, and it will have several other natural geometrical properties.  

Let us consider the space $\R^{<\omega}$ of $\omega$-sequences of real numbers that are eventually zero.  Each $\R^n$ appears canonically as the collection of sequences $\vec x$ such that $\vec x(m) = 0$ for all $m \geq n$.  Of course, this real vector space comes along with the standard Euclidean norm.


For a detailed discussion of the following facts of classical analysis, see Chapters 11 and 12 of \cite{Zorich}.
For a set $S \subseteq \R^n$, we say that $S$ is a \emph{parameterized ($k$-dimensional) smooth surface} if there are bounded open sets $U \subseteq V \subseteq \R^k$ such that the closure $\bar U$ of $U$ is contained in $V$, and there is an injective function $\varphi : V \to \R^n$ such that $S = \varphi[U]$ and $\varphi,\varphi^{-1}$ are both continuously differentiable (i.e.\ $C^1$).  The purpose of the set $V$ is simply to guarantee that $\varphi$ has a continuous derivative defined on a compact set.  If $S$ is such a surface, witnessed by $\varphi : U \to S$, then the classical volume of $S$ is given by
$$\vol_k(S) = \int_{U} \sqrt{\det(G_\varphi(\vec x))} \, d\vec x,$$
where $G_\varphi$ is the \emph{Gram matrix} of all inner products of partial derivatives of $\varphi$.  A key result is that this number does not depend on the way a surface is parameterized.  
\begin{fact}
\label{surfvol}
Suppose $\varphi_0 : U_0 \to S$ is a parameterization of a smooth surface $S$, and $\varphi_1: U_1 \to S$ is another parameterization.  Then
 $$ \int_{U_0} \sqrt{\det(G_{\varphi_0}(\vec x))} \, d\vec x =  \int_{U_1} \sqrt{\det(G_{\varphi_1}(\vec x))} \, d\vec x.$$
\end{fact}

If $\varphi : U \to S$ is a parameterization of a $k$-dimensional smooth surface $S$ and $A \subseteq U$ is Lebesgue measurable in $\R^k$, let us say $\varphi[A]$ is a \emph{measurable fragment} of $S$.  We can define the measure of $\varphi[A]$ to be the Lebesgue integral $\int_A \sqrt{\det G_\varphi} \, d\vec x$.  This measure is independent of parameterization.  For suppose $\psi : V \to S$ is another parameterization and $A \subseteq U' \subseteq U$, where $U'$ is open.  Then $\phi[U']$ is also a smooth surface, and $\psi^{-1} \circ \phi [U']$ is an open set $V' \subseteq V$.  By Fact \ref{surfvol}, $\int_{U'} \sqrt{\det G_\varphi} \, d\vec x = \int_{V'} \sqrt{\det G_\psi} \, d\vec x$.  Thus taking the infimum of these values over open cover covers of $A$ versus $\psi^{-1} \circ \phi [A]$ attains the same real number.  For any measurable fragment $A$ of a $k$-dimensional parameterized smooth surface, let $\vol_k(A)$ be this measure.  

Note that if $A,B,C$ are measurable fragments of $k$-dimensional parameterized surfaces, $A \cap B = \emptyset$, and $A \cup B = C$, then $\vol_k(C) = \vol_k(A) + \vol_k(B)$.  This is because for any parameterization $\varphi : U \to S \supseteq C$, 
\begin{align*}
\vol_k(C) = \int_{\varphi^{-1}[C]} \sqrt{\det G_\varphi} \, d\vec x &=\int_{\varphi^{-1}[A]} \sqrt{\det G_\varphi} \, d\vec x \,+ \int_{\varphi^{-1}[B]} \sqrt{\det G_\varphi} \, d\vec x \\
&= \vol_k(A) + \vol_k(B).
\end{align*}
Another important fact we will need is:
\begin{fact}
\label{lowerdim}
If $k<n$, $U \subseteq \R^k$ is open, and $\varphi : U \to \R^n$ is $C^1$, then the $n$-dimensional Lebesgue measure of $\varphi[U]$ is zero.
\end{fact}
It follows that for any smooth surface $S \subseteq \R^n$, there is at most one natural number $k$ such that $S$ is parametrizable in $k$ dimensions.  Furthermore, if $A \subseteq S$ is Borel and $T \subseteq \R^n$ is an $m$-dimensional smooth surface, where $m>k$, then $\vol_m(A \cap T) = 0$.

In general, smooth surfaces $S$ do not need to be parameterized by a single map, but rather they are given by a countable \emph{atlas}, $\{ \varphi_i : i \in \omega \}$, where each $\varphi_i$ is a parameterization of a smooth surface $S_i$, $S = \bigcup_i S_i$, and some differentiability conditions hold on the compositions $\varphi_i^{-1} \circ \varphi_j$.  For our purposes here, we will only consider surfaces given by a \emph{finite} atlas. This suffices for many applications, such as for compact surfaces.  But more generally, we ignore the coherence conditions between the parameterizations and consider \emph{piecewise smooth surfaces}, which are just finite unions of parameterized surfaces.

Suppose $S \subseteq \R^n$ is a $k$-dimensional piecewise smooth surface given as a finite union of parameterized smooth surfaces in two ways, $S = \bigcup_{i \leq n} S_i = \bigcup_{i \leq m} T_i$.  Let $A \subseteq S$ be Borel.  By putting $S'_i = A \cap S_i \setminus \bigcup_{j<i}S_j$ and $T'_i = A \cap T_i \setminus \bigcup_{j<i}T_j$, we present $A$ as a disjoint union of Borel fragments of parameterized surfaces in two ways.  Consider the set of all Boolean combinations of the $S'_i$ and $T'_i$, besides the complement of $A$, listed as $\{ B_i : i \leq N \}$.  Then for each $j \leq m,n$, it follows by the observations above that 
$\vol_k(S'_j) = \sum_{B_i \subseteq S'_j} \vol_k(B_i)$ and
$\vol_k(T'_j) = \sum_{B_i \subseteq T'_j} \vol_k(B_i)$.
Therefore,
$$\sum_{i = 0}^n \vol_k(S'_i) = \sum_{i = 0}^m \vol_k(T'_i) = \sum_{i=0}^N \vol_k(B_i).$$
This allows us to unambiguously define $\vol_k(A)$ as $\sum_{i=0}^M \vol_k(C_i)$, where $\{C_i \}_{i \leq M }$ is \emph{any} partition of $A$ into parameterized $k$-dimensional Borel surface fragments.  Furthermore, if $C$ is the disjoint union of $A$ and $B$, where each is a Borel subset of a $k$-dimensional piecewise smooth surface, then taking partitions of $A$ and $B$ into parameterized Borel fragments yields one for $C$, call it $\{ P_i \}_{i \leq N}$.  Since $A = \bigcup_{P_i \subseteq A} P_i$ and $B = \bigcup_{P_i \subseteq B} P_i$, it follows that $\vol_k(C) = \vol_k(A) + \vol_k(B)$.  In summary, we have:

\begin{proposition}
Let $k\leq n$ be positive natural numbers.  The function $\vol_k$ is a finitely additive measure on the Borel subsets of $k$-dimensional piecewise smooth surfaces contained in $\R^n$.
\end{proposition}

For a positive integer $n$, let $\mu_n$ be the Lebesgue measure on $\R^n$.
Let us call a set $A \subseteq \R^{<\omega}$ \emph{middling} if for all but finitely many $n<\omega$, $\mu_n(A \cap \R^n) < \infty$, and for infinitely many $n<\omega$, $\mu_n(A \cap \R^n) > 0$.  Intuitively, middling sets are larger than finite-dimensional sets but much smaller than the whole space.  Clearly, every open ball in $\R^{<\omega}$ is middling.

\begin{theorem}
There is a fine filter $\Gamma$ over $[\R^{<\omega}]^{<\omega}$ and a $\ll$-increasing sequence of positive infinitesimals $\la \e_i : i < \omega \ra \subseteq \pow(\R,\Gamma)$, such that, if 
$m(A) = \int \chi_A d\Gamma$ 
for $A \subseteq \R^{<\omega}$, then:
\begin{enumerate}
\item $\e_n = m([0,1]^n)$, the measure of the $n$-dimensional unit cube. 
\item For any measurable fragment $A$ of a $n$-dimensional piecewise smooth surface $S$, 
$$m(A)= \vol_n(A) \e_n + \delta,$$ 
where $\delta \ll \e_n.$
\item\label{countable} For any countable $C \subseteq \R^{<\omega}$, $m(C) \ll \e_1$.
\item For any middling Borel $A \subseteq \R^{<\omega}$ and any $\vec x \in \R^{<\omega}$, $m(A+\vec x) \approx m(A)$.
\end{enumerate}
\end{theorem}

\begin{proof}
	Let $\Gamma$ be generated by closing the following collection under intersections and supersets:
	\begin{enumerate}
	\item $\{ z : \vec x \in z \}$, for $\vec x \in \R^{<\omega}$.
	\item $\{ z : |z \cap [0,1]^n| > k|z \cap [0,1]^m| \}$, for natural numbers $n>m$ and $k$.
	\item $\{ z : \left| |z \cap A|/|z \cap [0,1]^k| - \vol_k(A) \right| < 1/m \}$ for each Borel subset $A$ of a $k$-dimensional piecewise smooth surface $S \subseteq \R^n$ and each integer $m>0$.
	\item\label{countgen} $\{ z : |z \cap [0,1]| > k|z \cap C| \}$ for each countable $C \subseteq \R^{<\omega}$ and integer $k$.
	\item $\{ z : \left| |z \cap A|/|z \cap (A+\vec x)| - 1 \right| < 1/n \}$ for each middling Borel $A \subseteq \R^{<\omega}$, $\vec x \in \R^{<\omega}$ and integer $n>0$.

	\end{enumerate}
A filter containing all of these sets clearly gives us what we want.  We must show that this family has the finite intersection property.

Suppose we are given finitely many points,
piecewise smooth surfaces with given Borel subsets,
 middling Borel sets, and countable sets.  Let $\e>0$ be arbitrary.
	Order the surfaces as 
	$$S^1_0,\dots,S^1_{n_1},S^2_0,\dots,S^2_{n_2},\dots,S^k_0,\dots,S^k_{n_k},$$
	where each $S^d_i$ is $d$-dimensional.  Let $A^d_i$ be the given Borel subset of $S^d_i$.  We may assume that for $1 \leq d \leq k$, $A^d_0 = [0,1]^d$.  
	
	Let $z_0$ be the given set of points, and let $C$ be the union of the given countable sets.  We inductively build a sequence of finite sets $z_0 \subseteq z_1 \subseteq \dots \subseteq z_k$ as follows.  Suppose we have $z_{d-1}$.  For $i \leq n_d$, let 
$B^d_i = (A^d_i \setminus C) \setminus \bigcup_{j<d ; r \leq n_j} S^j_r$.
By Fact \ref{lowerdim}, $\vol_d(B^d_i) = \vol_d(A^d_i)$.
By Lemma \ref{combinatorial}, there is a finite $z_d \supseteq z_{d-1}$ with the following properties:
\begin{enumerate}
\item\label{dwarf} $|z_{d-1}|/|z_d| < \e$;
\item\label{nodisturb} $z_d \setminus z_{d-1} \subseteq \bigcup_{i \leq n_d} B^d_i$;
\item\label{proportions} for $1 \leq i \leq n_d$, $\left| |z_d\cap B^d_i| / |z_d\cap [0,1]^d| - \vol_d(B^d_i)\right|\ < \e$.
\end{enumerate}
When we arrive at $z_k$, we have a set satisfying the desired inequalities related to the Borel subsets of smooth surfaces.  (\ref{dwarf}) goes towards making smaller dimensional surfaces infinitesimal relative to larger dimensional ones.   (\ref{nodisturb}) ensures that our work in higher dimensions does not disturb the proportions of (\ref{dwarf}) and (\ref{proportions}) set up for the lower dimensions.

Let $M_1,\dots,M_s$ be the given middling Borel sets.  Pick an increasing sequence of natural numbers $m_1 < \dots < m_s$ such that each $S^d_i \subseteq \R^{m_1-1}$, $z_0 \subseteq \R^{m_1}$, and for $1 \leq i,j \leq s$, $\mu_{m_i}(M_j \cap \R^{m_i}) < \infty$ and $\mu_{m_i}(M_i \cap \R^{m_i}) > 0$.  For $1 \leq i \leq s$, let $y_i$ be the collection of indices $j$ such that $\mu_{m_i}(M_j \cap \R^{m_i}) > 0$.  We inductively build a sequence of finite sets $z_k \subseteq z_{k+1} \subseteq \dots \subseteq z_{k+s}$.

Assume we have $z_{d-1}$, where $d > k$.  Consider the collection of all translations $M_j + \vec x$ for $\vec x \in z_0 \cup \{ \vec 0 \}$ and $j \in y_d$.  For $j \in y_d$, $M_j \cap \R^{m_d}$ has the same Lebesgue measure as $(M_j+ \vec x) \cap \R^{m_d}$.  By Lemma \ref{combinatorial}, we can select a finite $z_d \subseteq \R^{m_d}$ such that:
\begin{enumerate}
\item\label{propmiddling} for $j \in y_d$ and $\vec x \in  z_0$, $\left| \frac{|z_d \cap M_j | }{ |z_d \cap (M_j + \vec x)|} - 1 \right| < \e$;
\item\label{nomess} $(z_d \setminus z_{d-1}) \cap (C \cup \R^{m_d - 1} \cup \bigcup_{i \notin y_d} M_i ) = \emptyset$.
\end{enumerate}
To check that this works, let $j \leq s$ and let $d$ be the largest integer such that $j \in y_d$.  Then the desired inequalties hold for $z_d$  by (\ref{propmiddling}).  They are preserved for $z_{k+s}$ by (\ref{nomess}).  The fact that $C$ is mentioned in (\ref{nomess}) ensures that we preserve the smallness properties of $C$ in relation to the smooth surfaces as well.  

Now $z_{k+s}$ is a finite set which, with a small enough choice of $\e$, witnesses the finite intersection property of the collection of interest.
	\end{proof}
	
	\begin{remark}
If $\kappa$ is a cardinal such that every set of reals of size $<\kappa$ has Lebesgue measure zero, then we can replace ``countable'' with ``$<\kappa$-sized'' in item (\ref{countable}) of the theorem.  This just requires a corresponding adjustment in item (\ref{countgen}) of the definition of $\Gamma$.  Let us call the resulting filter $\Gamma_\kappa$.
\end{remark}

\subsection{The Borel-Kolmogorov paradox}\label{sec bk}

The Borel-Kolmogorov paradox concerns a violation of intuitions about conditional probability in the context of geometry on the two-dimensional sphere.  
Let us first discuss the paradox, following the more synthetic-geometrical presentation of Easwaran \cite{MR2712139}.

Consider a sphere $S$ with a given axis $a_0$ and a small circular region $A$ around one of the poles determined by $a_0$.  For example, $A$ could be the set of all points north of the arctic circle on the earth. Now consider the set $\mathcal C_0$ of all great circles touching the two ends of the axis $a_0$.  For reasons of symmetry, the conditional probability $\Pr(A | C)$, or the proporition of measures $m(A \cap C) / m(C)$, should be the same for all $C \in \mathcal C_0$.  Furthermore, this should be in the same proportion as $m(A)/m(S)$.  Now let $B$ be the surface of revolution obtained by revolving $A$ around an axis $a_1$ perpendicular to $a_0$. Then $B$ is of strictly larger surface area than $A$.  Let $\mathcal C_1$ be the collection of great circles touching the ends of $a_1$.  For the same reasons as before, for all $C \in \mathcal C_1$, $m(B \cap C)/m(C)$ should be the same as $m(B)/m(S)$.  But there is a $C^* \in \mathcal C_0 \cap \mathcal C_1$.  Thus we have 
$$m(A) / m(S) = m(A \cap C^*) / m(C^*) = m(B \cap C^*) / m(C^*) = m(B) / m(S),$$
and so $m(A) = m(B)$.  This is a contradiction.

Of course, the argument works equally well if we replace ``$=$'' with ``$\approx$'' in the case that $m$ is non-Archimedean, and we arrange that $m(A)/m(S) \not\approx m(B)/m(S)$.  Kolmogorov's diagnosis of the error in the paradox was, ``This shows that the concept of a conditional probability with regard to an isolated given hypothesis whose probability equals 0 is inadmissible'' \cite{MR0079843}.  On our view, this cannot be the the right explanation, since the paradox carries the same force if we use a non-Archimedean analysis that gives all nonempty sets a nonzero measure, as we have done.

On our view, the error lies in the claim that the conditional probabilities $\Pr(A | C^*)$ and $\Pr(B|C^*)$ ``should be'' in the same (or approximately the same) proportion as the sizes of the background sets $A$ and $B$ relative to the sphere. This sounds somewhat intuitive, but we contend that it is much more intuitive that $m(A \cap C^*)/m(C^*)$ should be the proportion of the arc length of $C^*$ taken up by $A$, as is the case for our filter-integral on $\R^{<\omega}$, without regard to the larger background of the set $A$.

So why ``should'' $m(A \cap C^*) / m(C^*)$ and $m(A ) / m(S)$ be the same?  The argument advanced by Easwaran \cite{MR2712139} is a principle he calls ``conglomerability."  This is a generalization of a formula for weighted averages from the finite to the infinite case.  If $A_0,\dots,A_n$ are disjoint sets with nonzero measure, then simple arithmetic implies that for any measurable $B \subseteq A_0 \cup \dots \cup A_n$, 
$$\Pr(B) = \Pr(B | A_0)\Pr(A_0) + \dots + \Pr(B | A_n)\Pr(A_n).$$
It is easy to see that if the probability measure is countably additive, then this generalizes to countable collections of disjoint measurable sets.  Easwaran generalizes this further to say that if $\{ A_i : i \in I \}$ is any partition of a set $A$, then for every $B \subseteq A$, we should have the integral equation:
$$ \Pr(B) = \int \left( \sum_{i} \Pr(B | A_i) \chi_{A_i}(x) \right) dx.$$
(Note that since the $A_i$ are pairwise disjoint, the sum in the integrand has at most one nonzero term at a given $x$.)  In the context of our reasoning about the sphere, the idea is that when the two poles are removed, the set of great circles through those poles forms a partition of the sphere.  Thus the proportion of the sphere taken up by the set $A$ should be the conglomeration of all of the pieces meeting the great circles.  Since all of these pieces are congruent, this is an integral of a constant function with value $c = \Pr(A|C^*)$.  In other words, assuming we start with a sphere with surface area 1, $\Pr(A) = \int c \, dx = c$.

Now we know this kind of equation will not hold in general, but it is interesting to look closely at what it says in the context of our filter integrals.  Suppose $F$ is a fine filter over $[X]^{<\omega}$, $B \subseteq X$, and $\{ A_i : i \in I \}$ is a partition of $X$ into nonempty sets.  Then by definition we have:
$$ \int \chi_B \, dF = \int \left( \sum_i \chi_{B \cap A_i} \right) dF = \left[ z \mapsto \sum_i \frac{|B \cap A_i \cap z|}{|z|} \right]_F$$
$$= \left[ z \mapsto \sum_i \frac{|B \cap A_i \cap z|}{| A_i \cap z|} \frac{|A_i \cap z|}{|z|} \right]_F $$
This looks a lot like we are integrating $\sum_i \Pr(B | A_i) \chi_{A_i}(x)$.  However, in our situation, $\Pr(B | A_i)$ is an integral and typically a nonstandard element of $\pow(\Q,F)$.  If it has a standard part, this value depends on the convergence properties modulo $F$, and we should not expect a similar-looking formula to be substitutable back into the process and have the convergence come out unaffected.

Easwaran ultimately comes down in favor of the conglomerability principle, and due to several reasons including the above paradox,  against the position that conditional probabilities should be construed as ratios of unconditional measures.  Instead, he argues that conditional probabilities depend on a context, namely a given partition of the underlying space.  However, we contend that the filter integral gives a coherent and natural picture of conditional probability as a ratio of measures for any nonempty condition, and the geometric intuitions buttressing this picture outweigh the philosophical arguments for conglomerability.

\subsection{Dimension}
\label{dimsec}

The above result suggests a relevant notion of dimension of an arbitrary subset $A$ of $\R^{<\omega}$ as the Archimedean equivalence class of $\int \chi_A \, d\Gamma$.  We would like to understand the structural relations among the $\Gamma$-dimensions.  The usual integer dimensions are ordered in the expected way, while middling sets have dimension larger than all of these, and the whole space is still of higher dimension than any middling set.  There are also dimensions in between.  For example, $(\Q \times \R) \cap [0,1]^2$ has dimension between 1 and 2.  Its measure is larger than any finite length curve since it contains infinitely many pairwise disjoint unit length line segments, but its 2-dimensional volume is zero.  Thus its Archimedean class is between those of $\e_1$ and $\e_2$.

Suppose $F$ is a fine filter over $[X]^{<\omega}$.  Let $\dim_F(A)$ denote the Archimedean class of $\int \chi_A \, dF$.  Let us say $\dim_F(A) < \dim_F(B)$ when $\int \chi_A \, dF \ll \int \chi_B \, dF$.
Note that if $F' \supseteq F$, then $\dim_F(A) < \dim_F(B)$ implies $\dim_{F'}(A) < \dim_{F'}(B)$.
Let us say that a set $A \subseteq X$ is \emph{$F$-solid} if for all $Y \subseteq X$ such that $|Y|<|X|$, $\dim_F(Y) < \dim_F(A)$.  If every set of reals of size less than the cardinality of the continuum $\frak c = 2^\omega$ has Lebesgue measure zero, then each positive-volume Borel subset of a finite-dimensional surface in $\R^{<\omega}$ is $\Gamma_{\frak c}$-solid.

Recall that \emph{Martin's Axiom} (MA) says that for any partial order $\mathbb P$ satsifying the countable chain condition (ccc), and any collection $\{ D_\alpha : \alpha<\kappa \}$ of dense subsets of $\mathbb P$, where $\kappa < \frak c$, there is a filter $G \subseteq \mathbb P$ such that $G \cap D_\alpha \not= \emptyset$ for each $\alpha<\kappa$.  MA is implied by the continuum hypothesis (CH), but $\neg$CH does not decide MA.  MA implies that $\frak c$ is a regular cardinal, $2^\kappa = \frak c$ for all $\kappa<\frak c$, and every set of reals of size $<\frak c$ has Lebesgue measure zero.  See \cite{jech} for background.

\begin{lemma}
\label{malem}
Assume MA.  Let $F$ be a fine filter over $[\frak c]^{<\omega}$ that is generated by a base of size $\frak c$.  Suppose $\{A_\alpha : \alpha < \frak c\}$ and $\{B_\alpha : \alpha < \frak c\}$ are collections of subsets of $\frak c$ such that each $B_\alpha$ is $F$-solid, and for all $\alpha,\beta<\frak c$, $\dim_F(A_\alpha)<\dim_F(B_\beta)$.  Then there is a filter $F' \supseteq F$ with a base of size $\frak c$ and an $F'$-solid $C \subseteq \frak c$ such that for all $\alpha,\beta<\frak c$,
$\dim_{F'}(A_\alpha) < \dim_{F'}(C) <  \dim_{F'}(B_\beta).$
\end{lemma}

\begin{proof}
Let $\la X_\alpha : \alpha < \frak c \ra$ be an enumeration of a base for $F$.  Let $\la M_\alpha : \alpha < \frak c \ra$ be a sequence of elementary submodels of $H_{\frak c^+}$ such that:
\begin{itemize}
\item For each $\alpha< \frak c$, $|M_\alpha|<\frak c$, $M_\alpha \cap \frak c$ is an ordinal, and $M_\alpha \in M_{\alpha+1}$.
\item For each limit $\lambda < \frak c$, $M_\lambda = \bigcup_{\alpha<\lambda} M_\alpha$.
\item $F, \{(A_\alpha,B_\alpha,X_\alpha) : \alpha < \frak c\} \in M_0$.
\end{itemize}

For a set $X$, let $\fun(X,2,{<}\omega)$ be collection of finite partial functions from $X$ to $2$.  We partially order these functions by putting $p \leq q$ when $p$ extends $q$.  It is well-known that this partial order has the ccc.  For the rest of the argument, let $\mathbb P = \fun(\frak c,2,{<}\omega)$.

\begin{claim}\label{dense}
Suppose $\delta<\frak c$, $s \in [\frak c]^{<\omega}$, and $n \geq 2$.  For $p \in \mathbb P$, let $C_p = \{ \beta \in \dom(p) : p(\beta) = 1 \}$.   Consider the set
\begin{align*}
D_{\delta,s,n} =	&\hspace{1mm} \{ p \in \mathbb P: 	\dom(p) \in \bigcap_{i \in s} X_i, \text{ and for all } i,j \in s \\
			&\hspace{1mm} n\left(|\dom(p) \cap A_i| + |\dom(p) \cap \delta|\right) < |C_p \setminus \delta | < n^{-1}| \dom(p) \cap B_j| \}.
\end{align*}
Then $D_{\delta,s,n}$ is dense.
\end{claim}
\begin{proof}
Let $p \in \mathbb P$ be arbitrary.  
Using the assumptions that each $B_\beta$ is solid and of larger $F$-dimension than each $A_\alpha$, find $z \in \bigcap_{i \in s} X_i$ such that $z \supseteq \dom(p)$, $|z| > 2|\dom(p)|$, and for all $\alpha,\beta \in s$,
$$2n^2(|s| |z \cap A_\alpha | + |z \cap \delta|) < |z \cap B_\beta|.$$
By fineness, we may assume the numbers on the lefthand sides are all positive.
If $m = n^2 \left( |\bigcup_{i \in s} z \cap A_i | + |z \cap \delta| \right)$, then $|z \setminus\dom(p)| > m$.
Then choose a set $C^* \subseteq z \setminus (\dom(p) \cup \delta)$ of size $\frac{m}{n}+1$, which is possible since $n \geq 2$ and $|z \cap\delta|<m/n$.  Define an extension $q$ of $p$ with $\dom(q)= z$ by putting $q(\gamma) = 1$ for $\gamma\in C^*$, and otherwise $q(\gamma) = 0$ for $\gamma \in z \setminus \dom(p)$.  Then $q \in D_{\delta,s,n}$.
\end{proof}

By MA, let $G_0$ be $\mathbb P$-generic over $M_0$, i.e.\ $G_0$ is a filter that meets every dense subset of $\mathbb P$ which lies in $M_0$.  $G_0$ can be thought of as a function from $M_0\cap\frak c$ to $2$.  Let $C_0 = \{ \gamma : G_0(\gamma)= 1 \}$.  Assume inductively that we have a sequence of sets $\la C_\alpha \subseteq M_\alpha : \alpha < \beta \ra$, with $C_\alpha \cap M_{\alpha'} = C_{\alpha'}$ for $\alpha'<\alpha$.  If $\beta$ is a limit, let $C_\beta = \bigcup_{\alpha<\beta} C_\alpha$.  If $\beta = \beta'+1$, let $G_\beta$ be $\mathbb P$-generic over $M_\beta$, and let 
$$C_\beta = C_{\beta'} \cup \{ \gamma : \gamma > M_{\beta'} \cap \frak c, G_\beta(\gamma) = 1\}.$$
Finally, we let $C = \bigcup_{\alpha<\frak c} C_\alpha$.

We want to show that for each $\delta<\frak c$, each $s \in [\frak c]^{<\omega}$, and each positive $n \in\N$, there is $z \in \bigcap_{i \in s} X_i$ such that for $\alpha,\beta \in s$,
$$n(| z \cap A_\alpha | + |z \cap \delta|) < |z \cap C | < n^{-1} |z \cap B_\beta  |.$$
To find such $z$, let $\alpha$ be large enough such that $s,\delta \in M_\alpha$.  Then $C \cap M_{\alpha+1} = C_{\alpha+1}$, and $C_{\alpha+1} = C_\alpha \cup \{ \gamma : \gamma > M_{\alpha} \cap \frak c, G_{\alpha+1}(\gamma) = 1\}$, where $G_{\alpha+1}$ is $\mathbb P$-generic over $M_{\alpha+1}$.  By Claim \ref{dense}, there is some $z \in M_{\alpha+1}\cap\bigcap_{i \in s} X_i$ 
such that for all $i,j \in s$, 
$$2n\left(|z \cap A_i| + |z \cap M_\alpha |\right) < |z \cap C_{\alpha+1} \setminus M_\alpha | < (2n)^{-1}| z \cap B_j|.$$
In particular, $n(|z \cap A_i| + |z \cap \delta |) < |z \cap C |$, and
$$|z \cap C| = |z \cap C \cap M_\alpha| + |z \cap C \setminus M_\alpha | \leq 2|z \cap C \setminus M_\alpha| <n^{-1} | z \cap B_j|.$$

This means that the following family has the finite intersection property:
\begin{itemize}
\item $\{ z : n| z \cap A_\beta | < |z \cap C |\}$ for $n<\omega$ and $\alpha<\frak c$;
\item $\{ z : n|z \cap C | < |z \cap B_\beta| \}$ for $n<\omega$ and $\beta<\frak c$;
\item $\{ z : n| z \cap \gamma | < |z \cap C |\}$ for $n<\omega$ and $\gamma<\frak c$;
\item $X_\delta$ for $\delta<\frak c$.
\end{itemize}
Let $F'$ be the generated filter.
Then $C$ is $F'$-solid, and for $\alpha,\beta< \frak c$,
$\dim_{F'}(A_\alpha) < \dim_{F'}(C) < \dim_{F'}(B_\beta)$.
\end{proof}

\begin{theorem}
Assume MA and $2^{\frak c} = \frak c^+$.  There is an extension of $\Gamma_{\frak c}$ to an ultrafilter $U$ such that for any collections $\mathcal S,\mathcal T \subseteq \R^{<\omega}$ of size at most $\frak c$ such that $\dim_U(S)<\dim_U(T)$ for $S \in \mathcal S$ and $T \in \mathcal T$ and each $T \in \mathcal T$ is $U$-solid, there is a $U$-solid $C$ such that $\dim_U(S) < \dim_U(C) < \dim_U(T)$ for all $S \in \mathcal S$ and $T \in \mathcal T$.

Consequently, for any sets $A,B$ such that $B$ is solid and $\dim_U(A) < \dim_U(B)$, the collection of dimensions of $U$-solid sets in the open interval $(\dim_U(A),\dim_U(B))$ does not have a coinitial or cofinal set of size $\frak c$.
\end{theorem}

\begin{proof}
 We will produce $U$ as an increasing union of filters $\la F_\alpha : \alpha < \frak c^+ \ra$, with $F_0 = \Gamma_{\frak c}$.  We assume inductively that each $F_\alpha$ has a base of size $\frak c$.

Let $\pi$ be a function on $\frak c^+$ such that for every triple $(X,\mathcal S,\mathcal T)$, where $X \subseteq  [\R^{<\omega}]^{<\omega}$ and $\mathcal S,\mathcal T \subseteq \p(\R^{<\omega})$ are sets of size $\frak c$, $\pi(\alpha) = (X,\mathcal S,\mathcal T)$ for unboundedly many $\alpha<\frak c^+$.
Suppose we are given $F_\alpha$, and for every $S \in \pi(\alpha)_1$ and $T \in \pi(\alpha)_2$,
$\dim_{F_\alpha}(S) < \dim_{F_\alpha}(T)$, and $T$ is $F_\alpha$-solid.
By Lemma \ref{malem}, there is a filter $F'_\alpha \supseteq F_\alpha$ with a base of size $\frak c$ and an $F_\alpha'$-solid set $C$ such that
$\dim_{F_\alpha'}(S) < \dim_{F_\alpha'}(C)< \dim_{F_\alpha'}(T)$
for $S \in \mathcal S$ and $T \in \mathcal T$.
Let $F_{\alpha+1}$ be the filter generated by $F'_\alpha$ together with either $\pi(\alpha)_0$ or its complement, according to whichever family has the finite intersection property.

Let $U = \bigcup_{\alpha<\frak c^+} F_\alpha$.  Suppose $\mathcal S,\mathcal T$ are as hypothesized.
Then then there is some $\alpha < \frak c^+$ such that 
$\dim_{F_\alpha}(S) < \dim_{F_\alpha}(T)$
for $S\in \mathcal S$ and $T \in \mathcal T$, and every $T \in \mathcal T$ is $F_\alpha$-solid.
Let $\beta \geq \alpha$ be such that $\pi(\beta)_1 = \mathcal S$ and $\pi(\beta)_2 = \mathcal T$.  Then at stage $\beta+1$, we obtain an $F_{\beta+1}$-solid set $C$ that separates the $F_{\beta+1}$-dimensions of $\mathcal S$ from those of $\mathcal T$.  This continues to hold for $\dim_U$.
\end{proof}

\subsection{Representing general multi-dimensional measures}

The Hausdorff measure is a well-known measure-theoretic construction which assigns to subsets $X \subseteq \R^n$ a family of (outer) measures $\mathcal H^\alpha(X)$, for real numbers $\alpha$, $0\leq\alpha\leq n$.  The idea is to give a generalization of volumes of smooth surfaces, which incorporates a general notion of dimension, for a wide class of subsets of $\R^n$.  A basic property of Hausdorff measure is that for any $X \subseteq \R^n$, there is at most one real $\alpha$ such that $0<\mathcal H^\alpha(X)<\infty$, while $\mathcal H^\beta(X) = \infty$ for $\beta <\alpha$, and $\mathcal H^\beta(X) = 0$ for $\beta >\alpha$.  If such a value $\alpha$ exists for a set $X$, then $\alpha$ is called the Hausdorff dimension of $X$.

There are many variations on the Hausdorff measure that all agree on smooth surfaces but disagree in general (see \cite[p.\ 63]{MR1730695} for 8 examples).  A related notion is Minkowski content, which gives finitely but not countably additive measures and agrees with the Hausdorff and Lebesgue measures in special cases.

Our filter-integral on $\R^{<\omega}$ can be thought as another generalization of the classical notion of volume in a rather different direction.  Aspects of our construction apply in an abstract setting that covers many of the families of measures discussed above.
A similar result about the Hausdorff measures has been obtained by Wattenberg with techniques of nonstandard analysis \cite{10.2307/2041916}.

\begin{theorem}
	Suppose $X$ is a set, $\mathcal A$ is an algebra of subsets of $X$, $(I,<)$ is a linear order, and $\{ \mu_i : i \in I \}$ is a family of functions on $\mathcal A$ satisfying the following conditions.
	\begin{enumerate}
		\item For each $i \in I$, $\mu_i$ is a finitely additive measure on $\mathcal A$ taking extended real values in $[0,\infty]$.
		\item For each $i<j$ in $I$ and each $Y \in \mathcal A$, $\mu_i(Y) \geq \mu_j(Y)$.
		\item For each $Y \in \mathcal A$, there is at most one $i \in I$ such that $0 < \mu_i(Y) < \infty$.
		\item For each $i \in I$, there is some $Y \in \mathcal A$ such that $0 < \mu_i(Y) < \infty$.
		\item For each $i \in I$ and $Y \in \mathcal A$, if $0 < \mu_i(Y)$, then $Y$ is an infinite set.
	\end{enumerate}
	Then there is a fine filter $F$ over $[X]^{<\omega}$ and a $\ll$-increasing sequence $\la \e_i : i \in I \ra \subseteq \pow(\R,F)$ such that for all $Y \in \mathcal A$ and $i \in I$ with $\mu_i(Y)<\infty$,
	$$\int \chi_Y \, dF = \mu_i(Y) \e_i + \delta,$$
	where $\delta \ll \e_i$.
\end{theorem}
\begin{proof}
	We will show that the following family of sets has the finite intersection property, so that it is the basis of a fine filter $F$ over $[X]^{<\omega}$ with the desired properties.
	\begin{enumerate}
		\item $\{ z : x \in z \}$, for $ x \in X$.
		\item $\left\{ z : \left|\frac{|z\cap Y|}{|z\cap \overline{Y}|} -
		\frac{\mu(Y)}{\mu(\overline{Y})}\right|\ <\ \frac{1}{m} \right\}$ whenever $\mu_i(Y),\overline{Y} \in \mathcal A$, $\mu_i,\mu_i(\overline{Y})<\infty$, and   $0<\mu_i(\overline{Y})$;
		\item $\left\{ z : \frac{|z \cap Y|}{|z \cap \overline{Y}|} > m \right\}$ whenever $\mu_i(Y),\overline{Y} \in \mathcal A$, $0<\mu_i(\overline{Y})<\infty$ and $\mu_i(Y)=\infty$.
	\end{enumerate}
	
	To this end, consider finitely many points $x_1, \ldots, x_n \in X$ and finitely many sets $Y_1, \ldots, Y_v \in \mathcal A$.
	Define also
	\begin{itemize}
		\item $i_1 = \min\{i \in I : \mu_i(Y_j) > 0 \text{ for some } j \leq v \}$;
		\item $i_{n+1} = \min\{i > i_n : \mu_i(Y_j) > 0 \text{ for some } j \leq v\}$.
	\end{itemize}
	Since $v \in \N$, we can order such indexes as $i_1, \ldots, i_k$ with $k \leq v$.
	Without loss of generality, suppose also that for every $j=1, \ldots, k$ there exists a set $\overline{Y}_{j}$ such that $0<\mu_{i_j}(\overline{Y}_j) <\infty$. In fact, if there is no such set in the original list $Y_1, \ldots, Y_v$, it is sufficient to add one set that satisfies the desired inequalities for every dimension $i_1, \ldots, i_k$. We can do so by hypothesis (4).
	
	Consider the finite set $\{x_1, \ldots, x_n\}$ and the elements of $\mathcal A_1 = \{ Y_j : \mu_{i_1}(Y_j) < \infty  \}$.
	By Lemma \ref{combinatorial} applied to $\mu_{i_1}$, there exists a finite set $z_1$ such that
	\begin{enumerate}
		\item
		$x_1,\ldots,x_k\in z_1$\,;
		\item
		for every $Y \in \mathcal A_1$,
		$\left|\frac{|z_1\cap Y|}{|z_1\cap \overline{Y}_{1}|} -
		\frac{\mu_{i_1}(Y)}{\mu_{i_1}(\overline{Y}_{1})}\right|\ <\ \frac{1}{m}.$
	\end{enumerate}
	
	We can repeat a similar argument for $i_2$ in order to obtain a suitable finite set $z_2$. In this case, however, we have to take into account that we want our finite set $z_2$ to satisfy hypothesis (3) of the basis of $F$: namely, $\frac{|z_2 \cap Y|}{|z_2 \cap \overline{Y}_1|} > m$ for every $Y \not \in \mathcal A_1$.
	Thus we replace $\{x_1, \ldots, x_k\}$ with $z_1$, $\mathcal A_1$ with $\mathcal A_2 = \{ Y_j : 0<\mu_{i_2}(Y_j) < \infty  \}$ and we apply Lemma \ref{combinatorial} to obtain a finite set $z_2$ that satisfies
	\begin{enumerate}
		\item
		$z_1\subseteq z_2$\,;
		\item $z_2 \setminus z_1 \subseteq \bigcup \mathcal A_2 \setminus \bigcup \mathcal A_1$;
		\item
		for every $Y \in\mathcal A_2$,
		$\left|\frac{|z_2\cap Y|}{|z_2\cap \overline{Y}_{2}|} -
		\frac{\mu_{i_2}(Y)}{\mu_{i_2}(\overline{Y}_{2})}\right|\ <\ \frac{1}{m};$
		\item for every $Y \in\mathcal A_2$,
		$\frac{|z_2\cap Y|}{|z_1|}\ >\ m.$
	\end{enumerate}
	
	The second condition ensures that the inequalities arranged for $z_1$ with respect to $\mathcal A_1$ continue to hold for $z_2$.
	We proceed in a similar way and obtain the sets $z_3, \ldots, z_k$ that satisfy analogous properties.
	The set $z_k$ satisfies the desired conditions stated at the beginning of the proof for $x_1, \ldots, x_n$ and $Y_1, \ldots, Y_v \in \mathcal A$.
	Thus we have proved that the family of sets (1)--(3) has the finite intersection property, so it generates a fine filter $F$ over $[X]^{<\omega}$.
\end{proof}

\section{Non-Archimedean integration}
\label{sec nai}
Besides being able to represent real-valued measures, the filter integral has also relevant applications in non-Archimedean integration. Recall that for arbitrary fields $\F$, developing a non-Archimedean integration is still an open problem, despite some positive results established for particular classes of such fields. For a survey of this topic, see the introduction of \cite{real_valued_measure_preprint}. For known limitations of non-Archimedean integration, see \cite{lpcivita,forthcoming}.

In a non-Archimedean field $k \supset \R$, the idea underlying the Riemann and Lebesgue integrals of defining integrable functions as those that can be approximated arbitrarily well with step functions has some drawbacks. The main issue is that convergence in $k$ is much more restrictive than convergence in $\R$, so that it is not even possible to approximate arbitrarily well polynomials over finite intervals.
For instance, it is well-known that for all $\varepsilon\in\R$, $\varepsilon>0$ there exists a step function $s_\varepsilon : [0,1]\to \R$ such that $\max_{x \in [0,1]} |x^2-s_\varepsilon| < \varepsilon$. From this argument it is easy to obtain that for all positive $\varepsilon\in\R$, there exists a step function $s_\varepsilon : [0,1]_k \to k$ such that $\max_{x \in [0,1]_k} |x^2-s_\varepsilon| < \varepsilon$. However, no step function over $[0,1]_k$ can approximate $x^2$ up to an infinitesimal precision.

In order to overcome this issue in the Levi-Civita field, some authors have suggested to enlarge the family of ``elementary functions'' from step functions to analytic functions, with partial success \cite{lpcivita,measure_civita_1,measure_civita_2}. 

In this section we discuss how the filter integral provides an alternative approach to non-Archimedean integration. We start by acknowledging that in the non-Archimedean setting the filter integral lacks some geometric properties, especially when dealing with integrals over sets of an infinitesimal length.
Then we discuss a general representation theorem that allows us to definably extend real-valued measures to non-Archimedean field extensions of $\R$, in a way that the family of integrable functions is richer than that obtained with different approaches.
Finally, we show that the $F$-integral can be decomposed in a meaningful way according to the skeleton group of $k$.

\subsection{A geometric limitation}
%

Let $k \supset \R$, $\varepsilon \in k$, $0 < \varepsilon \ll 1$, $X = [0,1]_{\F}$ and consider the function $f = \varepsilon^{-1}\chi_{[0,\varepsilon]}$.
%
Since eventually the function $z \mapsto \sum_{x \in z} \frac{f(x)}{|z|}$ assumes an infinite value, $\int f\ dF$ is infinite number, regardless of the filter $F$.

This is at odds with the geometric intuition that, if the filter $F$ is chosen in a way that $\sqint \chi_{[0,1/n)} \ dF = \frac{1}{n}$ for every $n \in \N$, we would expect that the $F$-integral of $f$ is $1$.

In order to overcome this limitation, it might be possible to define a ``Riemann-like'' integral of a function $f: k \to k$ in the following way.
Let $z \in [k]^{<\omega}$ and let $x_1 < x_2 < \ldots < x_{|z|}$ be its elements. Then define the Riemann-like integral of a function $f$ as $$\left[z \mapsto \sum_{i=1}^{|z|-1} (x_{i+1}-x_i)f(x_i) \right]_{F}$$
The choice of evaluating $f$ at the left endpoint of the interval $[x_i, x_{i+1}]$ can be replaced by evaluating $f$ at other points of the interval.
However, this approach suffers from the same drawback discussed for other non-Archimedean measures, namely that the class of functions that can be approximated by step functions up to an arbitrary precision is too narrow. 
Thus, we find that it is more convenient to work directly with the $F$-integral.

\subsection{A general representation result}\label{sec representation}
Despite the limitation discussed above, the $F$-integral allows us to lift measures over $\R^n$ to $k^n$ in a way that the family of integrable functions is preserved.

In order to state the next result, we need to generalize the notion of $S$-continuous function, used in nonstandard analysis, to arbitrary non-Archimedean field extensions of $\R$.

\begin{definition}
	Let $\F \supset \R$ be an ordered field, $X \subseteq \F^n_{fin}$ and $f:X \to \F^m$.
	We say that $f$ is \emph{standardizable} iff
	\begin{itemize}
		\item $f(x) \sim f(y)$ whenever $x,y\in X$ and $x \sim y$;
		\item $f(x)$ is finite for every $x \in X$.
	\end{itemize}
	If $f$ is standardizable, we define its standard part $\st f : \st X \to \R^m$ as $\st f(x) = \st(f(y))$ for any $y \in X$ satisfying $\sh y = x$.
\end{definition}

Notice that, contrary to $S$-continuous functions of nonstandard analysis, the standard part of a standardizable function need not be continuous.

\begin{proposition}\label{prop sh int}
	Let $\F \supset \R$ be an ordered field, 
	and let $\mu$ satisfy the hypotheses of Theorem \ref{probext}. 
	Then there exists a fine filter $F$ on $[\F^n_{fin}]^{<\omega}$ such that for every standardizable $f : k^n_{fin} \to\F$, if $\sh f$ 
	is a bounded $\mu$-measurable function, then $f$ has a standard $F$-integral and $\sqint f dF = \int \sh{f} \, d\mu$.
\end{proposition}
\begin{proof}
	
	Let $F'$ be a filter satisfying Theorem \ref{probext} for ${\mu}$.	
	Then for any $\varepsilon \in \R$, $\varepsilon > 0$, and for any finitely many $f_1,\ldots,f_m$ such that $f_i : \R^n \to \R$ is a bounded $\mu$-measurable function, and for any finitely many points $x_1,\ldots ,x_\ell$ in $\R^n$ there is $A \in F'$ such that for any $z \in A$, 
$\{x_1, \ldots, x_\ell\} \subseteq z$ and
	$$
		\left| \sum_{x \in z} \frac{\st f_i(x)}{|z|} - \int f_i\ d\mu \right| < \varepsilon.
	$$
	
Now let $x_1,\dots,x_\ell \in k^n_{fin}$, let $f_1,\dots,f_m$ be standardizable, and let $\e>0$ be a real number.  Let $A \in F'$ be given with respect to the functions $\st f_1,\dots,\st f_m$ and the points $\sh x_1,\dots,\sh x_\ell$, and let $z' \in A$.  Let $z \in [k^n_{fin}]^{<\omega}$ be such that
\begin{itemize}
		\item $x_1,\dots,x_\ell \in z$;
		\item $\sh z = z'$;
		\item for every $r, s \in z'$, $|\{ x \in z : \sh{x} = r \}| = |\{ x \in z : \sh{x} = s \}|$.
	\end{itemize}
	These conditions ensure that when we compute the average of $f_i$ over $z$, we get the same standard part as computing the average of $\st f_i$ over $z'$.  

Thus for standardizable $f$ and real $\e>0$, if $A_{f,\e}$ is the set of $z \in [k^n_{fin}]^{<\omega}$ such that $$
		-\e < \sum_{x \in z} \frac{f(x)}{|z|} - \int \st f_i\ d\mu  < \varepsilon,
	$$
then the collection of all $A_{f,\e}$, together with the sets $\{ z : x \in z \}$ for $x \in k^n_{fin}$, generates a filter $F$ as desired.
\end{proof}

This proof can be adapted to prove the non-Archimedean counterpart of Theorem \ref{ext}.

\begin{corollary}
	Let $\F \supset \R$ be an ordered field, 
	and let $\mu$ satisfy the hypotheses of Theorem \ref{ext}.
	Then there exists a countable partition $\vec P$ of $k^n_{fin}$ and a fine filter $F$ on $[\F^n_{fin}]^{<\omega}$ such that for every standardizable $f : k^n_{fin} \to\F$, if $\sh f$ is a $\mu$-integrable function, then $f$ has a standard $(F,\vec P)$-integral and $\sqint f \, d(F,\vec P) = \int \sh{f} \, d\mu$.
\end{corollary}

In order to assess the relevance of these results, we suggest a comparison with Proposition 3.16 of \cite{lpcivita} in the case $\F = \civita$, the Levi-Civita field.
Proposition 3.16 of \cite{lpcivita} shows that any real-valued function that is not locally analytic at almost every point of its domain does not have a measurable representative with respect to the non-Archimedean uniform measure developed by Shamseddine and Berz \cite{measure_civita_1, measure_civita_2}.

Conversely, Proposition \ref{prop sh int} applied with $\mu = \lambda$, the Lebesgue measure over $\R^n$, shows that it is possible to define a fine filter $F$ on $[\civita^n_{fin}]^{<\omega}$ and a countable partition $\vec P$ of $\civita^n_{fin}$ such that
\begin{itemize}
	\item $\mu\left(\Pi_{i = 1}^n [a_i,b_i]_\civita\right) \approx \Pi_{i = 1}^n\sh{(b_i-a_i)}$ for every finite $a_1, \ldots, a_n, b_1, \ldots, b_n \in \civita$, i.e.\ $\mu$ is infinitesimally close to the uniform measure of Berz and Shamseddine, and
	\item if $f$ is standardizable and $\sh{f}$ is Lebesgue integrable, then $f$ has a standard $(F,\vec P)$-integral, and moreover its $(F,\vec P)$-integral is infinitesimally close to the Lebesgue integral of $\sh{f}$.
\end{itemize}
Thus the filter integral allows for a broader family of functions with a well-defined standard integral. Using the minimal definable filter instead of a larger filter (such as an ultrafilter) has the benefit of not assigning a standard integral to functions whose standard part is not $\mu$-measurable. 
The possibility of meaningfully enlarging the functions with a standard integral might enable further applications to mathematical models in the spirit of the ones discussed in Section 5 of \cite{lpcivita} or in Sections 4.6 and 4.7 of \cite{real_valued_measure_preprint}.

Recall also that a non-uniform measure theory over non-Archimedean fields has not yet been developed. In contrast, the filter integral allows one to do so by extending real-valued measures, an improvement upon the state of the art in this field.

\subsection{The decomposition of the $F$-integral for arbitrary fields}\label{sec decomposition}
The value of the filter integral of a function that takes values in $k \supset \R$ can be characterized based on the skeleton group of $\F$.

Invoking the Axiom of Choice, the Hahn Embedding Theorem \cite{Hahn,Clifford} allows us to write the elements of any $\F \supset \R$ as generalized formal power series over an ordered group $\Gamma$ (often called the skeleton group of $\F$) with real coefficients, in a way that the exponents of the terms with nonzero coefficients form a well-founded subset of $\Gamma$. Let $\lambda : \F \to \Gamma$ be the valuation chosen in such a way that if $\lambda(x) > 0$, then $|x|\ll 1$. The subfield $\{ x \in \F : \lambda(x) = 0\}$ is isomorphic to $\R$.

Let $\varepsilon > 0$ be an infinitesimal such that $\lambda(\varepsilon) = 1$.
We can write every $y \in \F$ as $\sum_{i \in \Gamma} a_i \varepsilon^i$ with $a_i \in \R$. 
Similarly, every function $f: X \to \F$ decomposes as $\sum_{i \in \Gamma} f_i \varepsilon^i$, where each $f_i : X \to \R$. For every $i \in \Gamma$ define also $f^{< i} = \sum_{j < i} f_j \varepsilon^j$ and $f^{> i} = \sum_{j > i} f_j \varepsilon^j$.

We would like to describe $F$-integrals of $\F$-valued functions in terms of some “orders of magnitude” components. It is possible to do so based on the properties of the skeleton group $\Gamma$.
For an arbitrary skeleton group $\Gamma$, we can exploit the decomposition of $f$ as $f^{< i} + f_i + f^{>i}$ to obtain the following $F$-integral decomposition:
\begin{equation}\label{eqn integral decomposition}
\int f dF = \int f^{<i} dF + \seq{\varepsilon^i} \int f_i dF + \int f^{>i} dF.
\end{equation}
In the above decomposition, since $\left| f^{>i} \right| \ll \varepsilon^i$, $\int f^{>i} dF \ll \seq{\varepsilon^i}$.
If $f_i$ has a finite standard $F$-integral, $\int f_i dF = \sqint f_i dF + \delta$, with $\delta \ll 1$.
Thus
$$
\int f dF = \int f^{<i} dF + \seq{\varepsilon^i} \sqint f_i dF + \eta_i,
$$
where $\eta_i = \seq{\varepsilon^i}\delta + \int f^{>i} dF$ satisfies $\eta_i \ll \seq{\varepsilon^i}$.

Additionally, if there exists $i\in \Gamma$ such that $f_i(x) \ne 0$ for some $x\in X$ and $f^{<i}(x) = 0$ for every $x\in X$, the $F$-integral of $f$ can be expressed as
$$
\int f dF = \seq{\varepsilon^i} \sqint f_i dF + \eta_i,
$$
i.e.\ the $F$-integral of $f$ is the integral of its leading term plus another term of a smaller magnitude.

If $\Gamma = \Z$, it is possible to further refine the above decomposition.
Let $\F \supset \R$ have the skeleton group $\Z$ (e.g.\ the field of formal Laurent series).  Assume that the function $f$ is bounded in $k$ and each $f_i$ has a finite standard $F$-integral.  Let $n \in \Z$ be such that $|f| < \e^n$.  For any $m \geq n$, the decomposition \eqref{eqn integral decomposition} can be further refined as

\begin{align}
\notag
\int f dF &=  \sum_{i = n}^m \seq{\varepsilon^i} \int f_i dF + \int f^{>m} dF \\
\label{eqn integral decomposition 2}
&= \sum_{i = n}^m \seq{\varepsilon^i} \left( \sqint f_i dF + \delta_i \right) + \int f^{>m} dF.
\end{align}
Let us compare the contribution of the error term $\delta_i$ to higher degrees of $\varepsilon$.
Let $a_i = \sqint f_i dF$: the error term is the equivalence class of the function $z \mapsto \frac{\sum_{x \in z} f_i(x)}{|z|}-a_i$. If $\delta_i \ne 0$, then $\left| \frac{\sum_{x \in z} f_i(x)}{|z|}-a_i \right|$ is eventually a positive real number. As a consequence, $|\delta_i| \gg \varepsilon^n$ for every $n > 0$. Therefore $\varepsilon^i \gg \varepsilon^i|\delta| \gg \varepsilon^j$ for every $j > i$.
Thus, each term of the sum \eqref{eqn integral decomposition 2} operates on a different scale, with no arithmetic influence between scales:
$$
\ldots \gg \seq{\varepsilon^{i-1}}|\delta_{i-1}| \gg \seq{\varepsilon^i}|a_i| \gg \seq{\varepsilon^i}|\delta_i| \gg \seq{\varepsilon^{i+1}}|a_i| \gg \ldots
$$

\section{Product Spaces}
\label{products}

Suppose we have fine filters $F,G$ over $[X]^{<\omega},[Y]^{<\omega}$ respectively.  
We construct a fine filter $F \times G$ over $[X \times Y]^{<\omega}$ concentrating on the finite rectangles, the collection of which is naturally isomorphic to $[X]^{<\omega} \times [Y]^{<\omega}$.  We put sets into $F \times G$  essentially when for a large subset of the $Y$-axis, the cross-section along the $X$-axis is large.   More precisely, $F \times G$ is the set of $A \subseteq [X \times Y]^{<\omega}$ such that
$$\{ z_1 \in [Y]^{<\omega} : \{ z_0 \in [X]^{<\omega} : z_0 \times z_1 \in A \} \in F \} \in G.$$
It is straightforward to check that $F\times G$ is a filter.  Furthermore, if $F$ and $G$ are both ultrafilters, then so is $F \times G$.

This operation is not symmetric.  For suppose $X$ is an infinite set and $F$ is a fine filter over $[X]^{<\omega}$. For $z \in [X]^{<\omega}$, let $A_z$ be the set of finite $z' \supseteq z$, and let $A = \bigcup_{z} A_z \times  \{ z \}$.  Then for all $z$,  $\{ z' : z' \times z \in A \} = A_z \in F$ by fineness, and so $A \in F^2$. But for any $z \in [X]^{<\omega}$ and any $z' \supsetneq z$, $z \times z' \notin A$, so $\{ z' : z \times z' \in A \} \notin F$.  Thus switching the roles of horizontal and vertical cross-sections yields a different filter.

Suppose $K$ is a divisible Abelian group.  For functions $f : X \times Y \to K$, we can compute $\int f \, d(F \times G)$ as before.  But we can also compute in two steps.  For fixed $p \in Y$, we obtain a value in $\pow(K,F)$ by taking $\int f(x,p) \, dF$.  This gives a function from $Y$ to $\pow(K,F)$, which we denote by $\int f(x,y) \, dF$.  We can then compute $\int(\int f(x,y) \, dF) \, dG$.  To show that this yields the same result, let us establish a general fact about iterated reduced powers:

\begin{lemma}[Folklore]
Suppose $F,G$ are filters over sets $X,Y$ respectively.  
Let $\frak A$ be any algebraic structure.  
Then there is a canonical isomorphism 
$$\iota : \pow(\frak A,F \times G) \cong \pow(\pow(\frak A,F),G).$$
\end{lemma}
\label{iterationisom}

\begin{proof}
First note that there is a natural correspondence between the objects of these structures, before we compute equivalence classes.  The elements of the iterated reduced power are represented by functions from $Y$ to functions from $X$ to $\frak A$, which are coded by functions on pairs.

Suppose $\varphi(v_0,\dots,v_n)$ is an atomic formula in the language of $\frak A$.  Let $f_0,\dots,f_n$ be functions from $X \times Y$ to $\frak A$.  Then:
\begin{align*}
&\pow(\frak A,F \times G) \models \varphi([f_0]_{F \times G},\dots,[f_n]_{F \times G})   \\
&\Longleftrightarrow\,  \{ (x,y) : \frak \models \varphi(f_0(x,y),\dots,f_n(x,y)) \} \in F \times G  \\
&\Longleftrightarrow\,   \{ y : \{ x : \frak A \models \varphi(f_0(x,y),\dots,f_n(x,y)) \} \in F \} \in G \\
&\Longleftrightarrow\,  \{ y : \pow(\frak A,F) \models \varphi([f_0(x,y)]_F,\dots,[f_n(x,y)]_F) \} \in G \\
&\Longleftrightarrow\,  \pow(\pow(\frak A,F),G) \models   \varphi([[f_0(x,y)]_F]_G,\dots,[[f_n(x,y)]_F]_G).  \\
\end{align*}
Thus we may define an isomorphism $\iota$ by $[f]_{F \times G} \mapsto [[f]_F]_G$.
\end{proof}

Because of the above fact, we will abuse notation slightly and write $a = b$ for $a \in \pow(\frak A,F \times G)$ and $b \in \pow(\pow(\frak A,F),G)$ when we really mean that $\iota(a) = b$, where $\iota$ is the canonical isomorphism above.

\begin{lemma}
Suppose $K$ is a divisible Abelian group, $F,G$ are fine filters over $[X]^{<\omega},[Y]^{<\omega}$ respectively. 
Then for all $f : X \times Y \to K$,
$$\int f \, d(F \times G) = \iint f \, dF dG.$$
\end{lemma}

\begin{proof}  Since $F \times G$ concentrates on the set of finite rectangles $z_0 \times z_1$,
$$\int f \, d(F \times G) = \left[ \sum_{(x,y) \in z_0 \times z_1} f(x,y)/|z_0 \times z_1| \right]_{F \times G} $$
The isomorphism $\iota$ maps this to:
\begin{align*}
 \left[ \left[ \sum_{(x,y) \in z_0 \times z_1} \frac{f(x,y)}{|z_0 || z_1|} \right]_F \right]_G =& \left[ \int \left(\sum_{y \in z_1} f(x,y)/|z_1|\right) dF \right]_G \\
 =& \left[ |z_1|^{-1} \sum_{y \in z_1} \int f(x,y) \, dF \right]_G \\
 =& \iint f(x,y) \, dFdG. \qedhere
 \end{align*}
 \end{proof}
 
The key reason we introduced the notion of a comparison ring is that it makes the theory of iterated filter integration more elegant. Since a reduced power of a comparison ring is also a comparison ring, general facts about integrating functions taking values in comparison rings apply to each step of an iterated integral. The remainder of this section is devoted to exploring some facts about standard parts in iterated filter integrals.

\begin{proposition}
\label{prodchar}
Suppose $F,G$ are fine filters over $[X]^{<\omega},[Y]^{<\omega}$ respectively.  For $A \subseteq X$ and $B \subseteq Y$,
$$\sqiint^+ \chi_{A\times B} \, dFdG= \left(\sqint^+ \chi_A \, dF\right)\left(\sqint^+ \chi_B \, dG\right);$$
$$\sqiint^- \chi_{A\times B} \, dFdG= \left(\sqint^- \chi_A \, dF\right)\left(\sqint^- \chi_B \, dG\right).$$
\end{proposition}

\begin{proof}
First we claim that there are filters $F_u,F_\ell \supseteq F$ and $G_u,G_\ell \supseteq F$ such that
\begin{itemize}
\item $\sqint \chi_A \, dF_u =  \sqint^+ \chi_A \, dF$;
\item $\sqint \chi_A \, dF_\ell =  \sqint^- \chi_A \, dF$;
\item $\sqint \chi_B \, dG_u =  \sqint^+ \chi_B \, dG$;
\item $\sqint \chi_B \, dG_\ell =  \sqint^- \chi_B \, dG$.
\end{itemize}
Let us show the first point; the others are similar.  Let $r = \sqint^+ \chi_A \, dF$.
We claim that for all $D \in F$ and all $\e>0$, there exists $z \in D$ such that $|z \cap A|/|z| > r - \e$.
Otherwise, there is $D \in F$ and $\e > 0$ such that for all $z \in D$, $|z \cap A|/|z| \leq r - \e$, which would mean that $ \sqint^+ \chi_A \, dF < r$, a contradiction.
It follows that $F$ together with the sets $\{ z \in [X]^{<\omega} : \left| |z \cap A|/|z| - r \right| < \e \}$, for $\e > 0$, generates a filter $F_u$, and $\sqint \chi_A \, dF_u = r$. 
 
Note that $\chi_{A \times B}(x,y)= \chi_A(x)\chi_B(y)$.  By linearity, $\iint \chi_{A \times B}(x,y) \, dF_u dG_u = \int(\chi_B(y) \int \chi_A(x) \, dF_u) \, dG_u$.  
By Lemma \ref{finitestd}, $\int \chi_A \, dF_u = \sqint \chi_A \, dF_u + \e$,
where $\e$ is an infinitesimal of $\pow(\R,F_u)$.
Thus, 
$$\iint \chi_{A \times B}(x,y) \, dF_u dG_u = \left[ z \mapsto  \left(  \sqint \chi_A \, dF_u + \e \right) \frac{|z \cap B|}{|z|} \right]_{G_u}$$
If $\sqint \chi_A \, dF_u = 0$, then this is an infinitesimal of $\pow(\pow(\R,F_u),G_u)$, and so $\sqiint \chi_{A\times B} \, dF_u dG_u= 0$.  Suppose then that $\sqint \chi_A \, dF_u \not= 0$. 
For any real $\delta>0$, there is $D \in G_u$ such that $| \frac{|z \cap B|}{|z|} - \sqint \chi_B \, dG_u | < \delta$ for $z \in D$.  Since $\e$ is infinitesimal, it follows that for $z \in D$,
$$-\delta \sqint \chi_A \, dF_u< \left(  \sqint \chi_A \, dF_u + \e \right) \frac{|z \cap B|}{|z|} - \sqint \chi_A \, dF_u \sqint \chi_B \, dG_u < \delta \sqint \chi_A \, dF_u.$$
Hence, 
$\sqiint \chi_{A\times B} \, dF_u dG_u= (\sqint^+ \chi_A \, dF)(\sqint^+ \chi_B \, dG).$
This shows that   
$$\left(\sqint^+ \chi_A \, dF\right)\left(\sqint^+ \chi_B \, dG\right) \leq \sqiint^+ \chi_{A\times B} \, dF dG.$$
It remains to show the reverse inequality.  
Let $p,q$ be rational numbers such that $p > \sqint^+ \chi_A \, dF$ and $q > \sqint^+ \chi_B \, dG$.
Then
$$\iint \chi_{A \times B} \, dF dG = \int \left( \int \chi_A \, dF \right) \chi_B \, dG < \int p \chi_B \, dG < pq.$$
Thus $\inf \{ s \in \Q : s > \iint \chi_{A \times B} \, dF dG \} =  (\sqint^+ \chi_A \, dF)(\sqint^+ \chi_B \, dG).$  The argument for the lower integrals is entirely analogous.
\end{proof}

If $f : X \times Y \to Z$ is a function, let $\bar f : Y \times X \to Z$ be defined by $\bar f(y,x) = f(x,y)$.  Our iterated integrals on products of two spaces are defined to integrate in the leftmost variable first, and this operation $f \mapsto \bar f$ allows us to consider switching the order of integration in line with our conventions.  

\begin{theorem}
\label{fubini}
Suppose $\mu,\nu$ are countably additive probability measures on $X,Y$ respectively.
Then for all $\mu\times\nu$-integrable functions $f : X \times Y \to \mathbb R$, there are sets $A \subseteq X$ and $B\subseteq Y$ such that $\mu(A) = \nu(B) = 1$, and
$$\int f \, d(\mu\times\nu) = \sqint_{A \times B} f \, d(F_\mu \times F_\nu) = \sqint_{B \times A} \bar f \, d(F_\nu \times F_\mu).$$
\end{theorem}

\begin{proof}
Let $f : X \times Y \to \mathbb R$ be $\mu\times\nu$-integrable.  By Fubini's Theorem, we have:

\begin{enumerate}[(a)]
\item There are sets $A \subseteq X$ and $B\subseteq Y$ such that $\mu(A) = \nu(B) = 1$, and for all $(x_0,y_0)  \in A\times B$, $f(x,y_0)$ is $\mu$-integrable and $f(x_0,y)$ is $\nu$-integrable.
\item The functions $x \mapsto \int f(x,y) \, d\nu$ and $y \mapsto \int f(x,y) \, d\mu$ are integrable.
\item $\int f \, d(\mu\times\nu) = \iint f \, d\mu d\nu =  \iint \bar f \, d\nu d\mu = \int \bar f \, d(\nu\times\mu)$.
\end{enumerate}

Since $y \mapsto \int_A f(x,y) \, d\mu$ is $\nu$-integrable,
$$\int f \, d(\mu \times \nu) = \int \left(\chi_B(y) \int_A f(x,y) \, d\mu\right) d\nu = \sqint \left(\chi_B(y) \int_A f(x,y) \, d\mu\right) dF_\nu.$$
For all $y \in Y$ such that $x \mapsto f(x,y)$ is $\mu$-integrable, we have:
$$\sqint_A f(x,y) \, dF_\mu := \sqint \chi_A(x) f(x,y) \, dF_\mu= \int \chi_A(x) f(x,y) \, d\mu.$$
By Lemma \ref{finiteslice}, 
$$\sqint \left(\chi_B(y) \sqint_A f(x,y) \, dF_\mu \right)  dF_\nu =\sqint \left(\chi_B(y) \int_A f(x,y) \, dF_\mu \right)  dF_\nu.$$
Putting this together, we have the desired conclusion that 
\begin{align*}
\int f \, d(\mu\times\nu) &= \int_B\left(\int_A f(x,y) \, d\mu\right) d\nu = \sqint_B \left( \int_A f(x,y) \, d\mu \right)  dF_\nu \\
&= \sqint_B \left( \sqint_A f(x,y) \, dF_\mu \right) dF_\nu =  \sqint_B \left( \int_A f(x,y) \, dF_\mu \right) dF_\nu \\
&=\sqiint_{A \times B} f \, dF_\mu dF_\nu = \sqint_{A \times B} f \, d(F_\mu \times F_\nu).
\end{align*} 
By exactly the same argument, $\int \bar f \, d(\nu\times\mu) = \sqint_{B \times A} \bar f \, d(F_\nu \times F_\mu)$. 
\end{proof}

Unfortunately, the restriction to measure-one sets $A$ and $B$ in the above result cannot in general be avoided.  To see this, consider the function $f$ on the open unit square defined by
$$f(x,y) = \begin{cases}
1/x &\text{if } $y = 1/2$ \\
0 &\text{otherwise}
\end{cases}
$$
Since $f$ is nonzero only a set of Lebesgue measure zero, its Lebesgue integral is zero.  Suppose $F = F_\lambda$, where $\lambda$ is the Lebesgue measure on $(0,1)$.
Then $\int f(x,1/2) \, dF$ is a positive infinite number $a \in \pow(\mathbb R,F)$.  For all $z \in [(0,1)]^{<\omega}$ with $1/2 \in z$, $|z|^{-1} \sum_{y\in z} \int f(x,y) \, dF = a/|z|$, which is still infinite.  Thus $\iint f \, dF^2$ is infinite.  On the other hand, for all $y \in (0,1)$, $|z|^{-1} \sum_{x\in z} \bar f(x,y) \leq 1/y|z|$.  Thus for all $y \in (0,1)$, $\int \bar f(x,y) \, dF$ is infinitesimal, and thus so is $\iint \bar f \, dF^2$.

\begin{proposition}
\label{topfinite}
Suppose $n$ is a natural number and for $i<n$, $\tau_i$ is a compact topology on $X_i$ and $F_i$ is a fine filter over $[X_i]^{<\omega}$ such that all $\tau_i$-continuous functions into $\R$ have a standard $F_i$-integral. Then any $(\prod_{i<n} \tau_i)$-continuous function from $\prod_{i<n} X_i$ to $\R$ has a finite standard $(\prod_{i<n} F_i)$-integral.
\end{proposition}

\begin{proof}
It suffices to prove the claim for $n = 2$, since the general case then follows by induction. 
Since $\tau_0 \times \tau_1$ is compact, every continuous $f : X_0 \times X_1 \to \R$ is bounded.  For each $y \in X_1$, $x \mapsto f(x,y)$ is a $\tau_0$-continuous function on $X_0$.  Thus $\sqint f(x,y) \, dF_0$ exists and is finite for each $y \in X_1$.  

We claim that the function $y \mapsto \sqint f(x,y) \, dF_0$ is a $\tau_1$-continuous function on $X_1$.  
Let $y \in X_1$ and $\e>0$ be given.  Let $r = \sqint f(x,y) \, dF_0$.  By the compactness of $\tau_0 \times \tau_1$, there is a finite collection of open rectangles $\{ A_i \times B_i : i < m \}$ such that if $(a_0,b_0),(a_1,b_1) \in A_i \times B_i$, then $| f(a_0,b_0) - f(a_1,b_1) | < \e$.  Let $B = \bigcap \{ B_i : y \in B_i \}$.  Then for all $y' \in B$ and all $x \in X_0$, $|f(x,y') - f(x,y)| < \e$.  It follows that $-\e < \int (f(x,y') -  f(x,y)) \, dF_0 < \e$.  Thus $y' \in B$ implies $|\sqint f(x,y') \, dF_0 - r| < \e$.

By hypothesis, $\int( \sqint f(x,y) \, dF_0) dF_1$ has a standard part.  It is finite since $f$ is bounded.  Applying Lemma \ref{finiteslice}, we get
$$\sqint \left( \sqint f(x,y) \, dF_0 \right) dF_1 = \sqint \left( \int  \ f(x,y) \, dF_0 \right) dF_1 = \sqint f \, d(F_0 \times F_1). \qedhere$$
\end{proof}

\section{Transfinite integrals}
\label{transfinite}
Let $(L,<)$ be a linear order, and let $\la (Z_i,F_i) : i \in L \ra$ be such that each $F_i$ is a filter over $Z_i$.  For a finite $a \subseteq L$, we interpret the product filter $\prod_{i \in a} F_i$ as taken in the order given by $(L,<)$.
\begin{lemma}
Suppose $a \subseteq b$ are finite nonempty subsets of $L$.  Let $\pi_{b,a} : \prod_{i \in b} Z_i \to \prod_{i \in a} Z_i$ be the canonical projection.  Then $A \in \prod_{i \in a} F_i$ if and only if $\pi_{b,a}^{-1}[A] \in \prod_{i \in b} F_i$.
\end{lemma}
\begin{proof}
Write $b$ in $L$-increasing order as $\xi_0 < \dots < \xi_{n-1}$.  Let $i_0 = \min\{i : \xi_i \in a \}$.  For $j \leq n$ and $x = a,b$, let $x_j = x \cap \{ \xi_0,\dots,\xi_{j-1}\}$.  We will show by induction that the conclusion holds with respect to $\pi_{b_j,a_j}$ for $i_0 < j \leq n$.  

Assume the claim holds for $i < j$.  Suppose first that $\xi_j \in a$.  Then:
\begin{align*}
A \in \prod_{i \in a_{j+1}} F_i	
	&\Longleftrightarrow \{ y : \{ \vec x : \vec x ^\frown \la y \ra \in A \} \in \prod_{i \in a_j} F_i \} \in F_{\xi_j} \\
	&\Longleftrightarrow \{ y : \{ \vec z : \pi_{b_j,a_j}(\vec z) ^\frown \la y \ra \in A \} \in \prod_{i \in b_j} F_i \} \in F_{\xi_j} \\
	&\Longleftrightarrow \{ \vec z : \pi_{b_{j+1},a_{j+1}}(\vec z) \in A \} \in \prod_{i \in b_{j+1}} F_i
\end{align*}
Suppose next that $\xi_j \notin a$.  Then by induction,
$$A \in \prod_{i \in a_{j+1}} F_i \Longleftrightarrow A \in \prod_{i \in a_{j}} F_i  \Longleftrightarrow \pi_{b_j,a_j}^{-1}[A] \in \prod_{i \in b_j} F_i$$
But $\pi_{b_{j+1},a_{j+1}}^{-1}[A] = \pi_{b_j,a_j}^{-1}[A] \times Z_{\xi_j}$, which is in $\prod_{i \in b_{j+1}} F_i$ if and only if $\pi_{b_j,a_j}^{-1}[A] \in \prod_{i \in b_j} F_i$.
\end{proof}

Now let $\frak A$ be any structure.  For finite $a \subseteq b$ contained in $L$, define a map $e_{a,b} : \pow(\frak A,\prod_{i \in a} F_i) \to \pow(\frak A,\prod_{i \in b} F_i)$ by $[f]_{\prod_{i \in a} F_i} \mapsto [f \circ \pi_{b,a}]_{\prod_{i \in b} F_i}$.  
The above lemma implies that each $e_{a,b}$ is an embedding.  
If $a \subseteq b \subseteq c$ are finite subsets of $L$, then $\pi_{c,a} = \pi_{b,a} \circ \pi_{c,b}$, and thus $e_{a,c} = e_{b,c} \circ e_{a,b}$.  
We define the direct limit of this system as the collection of equivalence classes of pairs $(a,x)$, where $a \subseteq L$ is finite and $x \in \pow(\frak A,\prod_{i \in a} F_i)$, with the equivalence relation $(a,x) \sim (b,y)$ holding when $e_{a,a \cup b}(x) = e_{b,a \cup b}(y)$.  We interpret the relation and function symbols in the language of $\frak A$ according to their interpretations in the finite iterated reduced powers, which is coherent because of the commuting system of embeddings.  Call this structure $\pow(\frak A,\vec F)$.  We have that for any atomic formula $\varphi(v_0,\dots,v_n)$ and any objects $[(a_0,x_0)],\dots,[(a_n,x_n)]$, if $b = a_0 \cup \dots \cup a_n$, then
$$\pow(\frak A,\vec F) \models  \varphi([(a_0,x_0)],\dots,[(a_n,x_n)])$$
$$\Longleftrightarrow \pow(\frak A,\prod_{i \in b} F_i ) \models \varphi(e_{a_0,b}(x_0),\dots,e_{a_n,b}(x_n))$$
$$\Longleftrightarrow \pow(\frak A,\prod_{i \in c} F_i ) \models \varphi(e_{a_0,c}(x_0),\dots,e_{a_n,c}(x_n)),$$
where $c \subseteq L$ is an arbitrary finite superset of $b$.  For finite $a \subseteq L$, let $e_a : \pow(\frak A,\prod_{i \in a} F_i) \to \pow(\frak A,\vec F)$ be the canonical embedding $x \mapsto [(a,x)]$.

\begin{lemma}
Suppose $K$ is a comparison ring, $L$ is a linear order, and $\la F_i : i \in L \ra$ is a sequence of filters.  Then $\pow(K,\vec F)$ is a comparison ring.  Furthermore, if $a \subseteq L$ is finite and $x \in \pow(K,\prod_{i \in a} F_i)$ has standard part $r$, then $\st(e_a(x)) = r$.
\end{lemma}

\begin{proof}[Proof (sketch)]
Lemma \ref{comparisonpower} implies that for each finite $a \subseteq L$, $\pow(K,\prod_{i \in a} F_i)$ is a comparison ring.  We need only to check that the axioms are preserved under embeddings.  For those that can be written $\Pi_1$-formulas, this is immediate.  Beyond this, the key is just that additive and multiplicative inverses are the unique solutions to equations like $a + x = 0$ and $ax= 1$, which are themselves atomic formulas.

For the second claim, note that for any $q_0,q_1 \in \Q$, the formulas ``$q_0 < x$'' and ``$x< q_1$'' are preserved by the embedding $e_a$.
\end{proof}

If $\vec X = \la X_i : i \in I \ra$ is a sequence of sets and $f : \prod \vec X \to Y$ is a function, let us say that $f$ is \emph{finitely dependent} when there is a finite $s \subseteq I$ such that whenever $\vec x,\vec y \in \prod\vec X$ are such that $\vec x \restriction s = \vec y \restriction s$, then $f(\vec x) = f(\vec y)$.  If $s_0,s_1$ both witness that $f$ is finitely dependent, then so does $s= s_0 \cap s_1$.  For suppose $\vec x \restriction s = \vec y \restriction s$, and put $\vec z = \vec x \restriction s_0 \cup \vec y \restriction (I \setminus s_0)$.  Then $f(\vec x) = f(\vec z) = f(\vec y)$.  Thus if $s_0,s_1$ are $\subseteq$-minimal witnesses to the finite dependency of $f$, then $s_0 = s_1$.  Thus let us define $\dep(f)$ as the smallest $s$ witnessing that $f$ is finitely dependent. $f$ is constant if and only if $\dep(f) = \emptyset$.  If $f$ is finitely dependent, then it canonically determines a function $f'$ on $\prod_{i \in J} X_i$, whenever $\dep(f) \subseteq J \subseteq I$, by putting $f'(\vec x) = f(\vec x \cup \vec y)$, where $\vec y \in \prod_{i \in I \setminus J}X_i$ is arbitrary.  We will abuse notation slightly and denote such $f'$ also by $f$.


\begin{proposition}
\label{welldef}
Suppose $L$ is a linear order, $G$ is a divisible Abelian group, $\la X_i : i \in L \ra$ is a sequence of sets, and for each $i \in L$, $F_i$ is a fine filter over $[X_i]^{<\omega}$.
Suppose $f : \prod\vec X \to G$ is finitely dependent and $a \supseteq \dep(f)$ is a finite subset of $L$.  Then
$ \int f \, d(\prod_{i \in a} F_i) =  e_{\dep(f),a}\left( \int f \, d(\prod_{i \in \dep(f)} F_i) \right)$.
\end{proposition}
\begin{proof}
For any finite $s \subseteq L$, $\prod_{i \in s} F_i$ can be regarded as a fine filter over $[\prod_{i \in s} X_i]^{<\omega}$ concentrating on the finite rectangles $\prod_{i \in s} z_i \subseteq \prod_{i \in s} X_i$.
For $s \supseteq \dep(f)$, let $g_s : [\prod_{i \in s} X_i]^{<\omega} \to G$ be defined by $g_s(z) = 0$ if $z$ is not a rectangle, and otherwise:
$$g_s\left(\prod_{i \in s} z_i\right) = \frac{1}{\prod_{i \in s} |z_i|} \sum_{\vec x \in \prod_{i \in s} z_i} f(\vec x)$$
Let $a \setminus \dep(f) = \{ i_0,\dots,i_{n} \}$.  In the expression above for $g_a$, for each $\vec y \in \prod_{i \in \dep(f)} z_i$, $f(\vec y)$ is repeated $|z_{i_0}| \cdots |z_{i_n} |$-many times and then divided by the same number.  So, for each rectangle $\prod_{i \in a} z_i$, $g_a(\prod_{i \in a} z_i) = g_{\dep(f)}( \prod_{i \in \dep(f)} z_i)$.  In other words, $g_a = g_{\dep(f)} \circ \pi_{a,\dep(f)}$.  Thus:
\begin{align*}
\int f \, d(\prod_{i \in a} F_i)	&= [g_a]_{\prod_{i \in a} F_i} = [g_{\dep(f)} \circ \pi_{a,\dep(f)}]_{\prod_{i \in a} F_i} \\
						&= e_{\dep(f),a}\left([g_{\dep(f)}]_{\prod_{i \in \dep(f)} F_i}\right) \\
						&= e_{\dep(f),a}\left( \int f \, d(\prod_{i \in \dep(f)} F_i) \right) \qedhere
\end{align*}
\end{proof}

Suppose $f : \prod\vec X \to G$ is finitely dependent. We define:
$$\int f \, d\vec F := e_{a}\left( \int f \, d(\prod_{i \in a} \vec F_i) \right) \in \pow(G,\vec F),$$
where $a$ is any finite superset of $\dep(f)$.  By the previous proposition, this is well-defined.

\begin{proposition}
Suppose $L$ is a linear order, $R$ is a ring, $\la X_i : i \in L \ra$ is a sequence of sets, and for each $i \in L$, $F_i$ is a fine filter over $[X_i]^{<\omega}$.
Any algebraic operation between finitely dependent functions on $\prod \vec X$ yields a finitely dependent function.
If $f,g : \prod \vec X \to R$ are finitely dependent and $r,s\in R$, then $\int (rf+sg) \, d\vec U = r\int f \, d\vec U + s\int g \, d\vec U$. (where we identify elements of $R$ with constant functions taking those values).
\end{proposition}
\begin{proof}For the first claim, just note that if the operation involves finitely many functions, then coordinates outside the union of their dependency sets have no influence.  For the second claim, let $a = \dep(f)$, let $b = \dep(g)$, and let $c = a \cup b$.  
Since $\int (rf+sg) \, d(\prod_{i \in c} F_i) = r\int f \, d(\prod_{i \in c} F_i) + s\int g \, d(\prod_{i \in c} F_i)$, 
the conclusion follows by the fact that $e_c$ is an embedding.
\end{proof}

Now we wish to extend the integrals $\int f \, d\vec F$ to give a value to \emph{all} functions on $\prod \vec X$ taking values in a divisible Abelian group $G$, not just the finitely dependent ones.
Choose a filter $H$ over $[L]^{<\omega} \times \prod\vec X$ which is fine in the sense that for every $i \in L$, $\{ (s,\vec x) : i \in s \} \in H$.  For each $f : \prod \vec X \to G$, each $s \in [L]^{<\omega}$, and each $\vec y \in  \prod \vec X$, we define a finitely dependent function:
$$f_{s,\vec y}(\vec x) = f( \vec x \restriction s \cup \vec y \restriction (L \setminus s))$$
We define the following operator on functions $f : \prod \vec X \to G$:
$$\int f \, d(\vec F,H) = \left[ (s,\vec y) \mapsto \int f_{s,\vec y} \, d\vec F \right]_H$$
Note that this operation enjoys the usual linearity properties.
If $f$ is finitely dependent, then for all $s \supseteq \dep(f)$ and all $\vec x,\vec y \in \prod \vec X$, $f(\vec x) = f_{s,\vec y}(\vec x)$.  Thus if $e : \pow(G,\vec F) \to \pow(\pow(G,\vec F),H)$ is the canonical embedding, then $e(\int f \, d\vec F) = \int f \,  d(\vec F,H)$.  


Let us say that a function $f : \prod \vec X \to \mathbb R$ is \emph{uniformly continuous} if for all $n \in \mathbb N$, there is a finite $s \subseteq L$ such that $\vec x \restriction s = \vec y \restriction s$ implies $| f(\vec x) - f(\vec y) |< 1/n$.
The next result shows that the standard integral of uniformly continuous functions depends only on the sequence of filters $\vec F$.

\begin{lemma}
\label{unifcont}
Suppose $\vec X,\vec F$ are $L$-sequences of sets and filters as above.
Suppose $f : \prod \vec X \to \mathbb R$ is uniformly continuous, and for $(s,\vec y) \in [L]^{<\omega} \times \prod\vec X$, $\int f_{s,\vec y} \, d(\prod_{i \in s} F_i)$ has a  standard part.
Then there is an $r \in \mathbb R \cup \{\pm\infty\}$ such that for all fine filters $H$ over $[L]^{<\omega} \times \prod \vec X$,
$\sqint f \, d(\vec F,H) = r$.
\end{lemma}

\begin{proof}
For $n \in \mathbb N$, let $s_n \in [L]^{<\omega}$ be such that  $| f(\vec x) - f(\vec y) | < 1/n$ whenever $\vec x \restriction s_n = \vec y \restriction s_n$.
Thus for all finite $t_0,t_1\supseteq s_n$ and all $\vec x,\vec y_0,\vec y_1 \in \prod \vec X$,
$$| f_{t_0,\vec y_0}(\vec x) - f_{t_1,\vec y_1}(\vec x) | < 1/n.$$
For $t = t_0 \cup t_1$, we have that
$$-1/n < \int f_{t_0,\vec y_0} \, d(\prod_{i \in t}F_i) - \int f_{t_1,\vec y_1} \, d(\prod_{i \in t}F_i) < 1/n,$$
and it follows by the fact that $e_t$ is an embedding that 
$$-1/n< \int f_{t_0,\vec y_0} \, d\vec U - \int f_{t_1,\vec y_1} \, d\vec U < 1/n.$$
If $\sqint f_{s,\vec y} \, d\vec F = \pm\infty$ for some $s,\vec y$, then $\sqint f_{t,\vec z} \, d\vec F = \pm\infty$ for all $t \supseteq s \cup s_1$ and all $\vec z$.  In this case, let $r = \pm\infty$ accordingly.  Otherwise, for each $n \in \mathbb N$,
$$B_n = \left\{  \sqint f_{t,\vec y} \, d\vec F : t \supseteq s_n \text{ and } \vec y \in \prod \vec X \right\}$$
is a subset of $\mathbb R$ of diameter $\leq 1/n$, and $B_{n+1} \subseteq B_n$.  There is a unique $r \in \mathbb R$ such that for all $n$, $\inf B_n \leq r \leq \sup B_n$.

Now let $H$ be a fine filter as hypothesized.  If $r$ is finite, then for each $t \supseteq s_n$ and each $\vec y \in \prod\vec X$,
$-1/n < \int f_{t,\vec y} \, d\vec F - r  < 1/n$.
Thus $\sqint f \, d(\vec F,H) = r$.  This also holds if $r$ is infinite by the remarks above.
\end{proof}

\begin{theorem}
\label{compactproduct}
Suppose $\vec X,\vec F$ are $L$-sequences of sets and filters as above.  Suppose for each $\alpha \in L$, $X_\alpha$ carries a compact topology $\tau_\alpha$, such that every $\tau_\alpha$-continuous function has a standard $F_\alpha$-integral.
Let $\tau$ be the product topology on $\prod \vec X$.  Then for every $\tau$-continuous $f : \prod \vec X \to \mathbb R$, there is a real $r$ such that $\sqint f \, d(\vec F,H) = r$ for every choice of $H$.
\end{theorem}
\begin{proof}
By Tychonoff's Theorem, the space $(\prod\vec X,\tau)$ is compact.  Let $f : \prod \vec X \to \R$ be $\tau$-continuous. 
For any $n \in \mathbb N$ and $\vec x\in\prod\vec X$, the inverse image of $(f(\vec x)-1/2n,f(\vec x)+1/2n)$ is open.
By compactness, there is a finite set $\{ \vec x_0,\dots,\vec x_n \} \subseteq \prod\vec X$ and, for each $i \leq n$, a finite collection of basic open sets $\{ A^{i,0},\dots,A^{i,{m_i}} \}$ such that 
$\prod\vec X = \bigcup_{i \leq n, j \leq m_i} A^{i,j}$,
and whenever $y \in A^{i,j}$, then $|f(\vec x_i) - f(\vec y)| < 1/2n$.  
For each $i \leq n$ and $j \leq m_i$, there is a finite $s^{i,j} \subseteq L$ such that $A^{i,j} = \prod_{\alpha\in L} B^{i,j}_\alpha$, where $B^{i,j}_\alpha \in \tau_\alpha$ for $\alpha \in s^{i,j}$, and otherwise $B^{i,j}_\alpha = X_\alpha$.  
Let $s = \bigcup_{i \leq n, j \leq m_i}  s^{i,j}$.  For $\vec y,\vec z \in \prod\vec X$, if $\vec y \restriction s = \vec z \restriction s$, then there are $i,j$ such that $\vec y,\vec z \in A^{i,j}$.  Thus $|f(\vec y) - f(\vec z)| < 1/n$, and $f$ is uniformly continuous.

Now suppose $(s,\vec y) \in [L]^{<\omega}\times \prod\vec X$.  Then $f_{s,\vec y}$ is a continuous function on $(\prod_{i \in s} X_i, \prod_{i \in s} \tau_i)$.  By Proposition \ref{topfinite}, $\sqint f_{s, \vec y} \, dF_i$ exists and is finite.
Lemma \ref{unifcont} implies that there is an $r$ such that $\sqint f \, d(\vec F,H) = r$ for every choice of $H$.  Since $f$ is bounded, $r$ is finite.
\end{proof}

As an application, we give a representation of the Lebesgue integral on the Cantor space $2^{\mathbb N}$ that is more ``inevitable'' than the representations of \S \ref{reps}.  In our context, let $L = \mathbb N$ with the usual ordering, and for each $i \in \mathbb N$, let $X_i = \{ 0,1 \}$ with the discrete topology.  Let $F_i$ be the unique fine filter over $\p(X_i)$, i.e.\ $A \in F_i$ if and only if $\{0,1\} \in A$.  For each $n \in \mathbb N$, integrals using $F_0 \times \dots \times F_{n-1}$ are the same as computing expected values with the uniform probability measure on a space with $2^n$ elements, or in other words, just finding the arithmetic average value of the function over all points.  Thus if $f : 2^{\mathbb N} \to \mathbb R$ is finitely dependent, then $\int f \, d\vec F = \int f \, d\lambda$, where $\lambda$ is the Lebesgue measure on $2^{\mathbb N}$.

By compactness, every continuous $f : 2^{\mathbb N} \to \mathbb R$ is uniformly continuous.  Thus for every $n \in \mathbb N$, there is a finitely dependent $g : 2^{\mathbb N} \to \mathbb R$ such that $|f(\vec x) - g(\vec x)| < 1/n$ for all $\vec x$.  It follows that $|\int f \, d\lambda - \int g \, d\lambda| < 1/n$.  Also, for every choice of the filter $H$, $-1/n< \int f\, d(\vec F,H) - \int g \, d(\vec F,H) < 1/n$.  Since $\int g \, d(\vec F,H) = \int g\, d\lambda$ and $n$ is arbitrary, $\sqint f \, d(\vec F,H) = \int f \, d\lambda$ for every choice of $H$.  

\begin{theorem}
\label{lebesgue}
Let $\lambda$ be the Lebesgue measure on the Cantor space $2^{\mathbb N}$.
For each $i \in \mathbb N$, let $X_i = 2$ and let $F_i$ be the unique fine filter over $\p(\{0,1\})$.
There is a definable fine filter $H$ over $[\mathbb N]^{<\omega} \times 2^{\mathbb N}$ such that for every Lebesgue-integrable function $f : 2^{\mathbb N} \to \mathbb R$, $\sqint f \, d(\vec F,H) = \int f \, d\lambda$.  $H$ is minimal among filters with this property.
\end{theorem}

We give a short proof using a result of Jessen \cite{MR1503159}:
\begin{theorem}[Jessen]
Suppose $f$ is a Lebesgue-integrable function on the unit interval.  Let 
$S_{f,n}(x) = n^{-1} \sum_{i = 0}^{n-1} f(x+i/n).$
Then $\lim_{n \to \infty} S_{f,2^n}(x) = \int f \, d\lambda$ for almost all $x$.
\end{theorem}

We note that by a result of W.\ Rudin \cite{MR159918}, we cannot simply replace $2^n$ with $n$ in the limit.

\begin{proof}[Proof of Theorem \ref{lebesgue}]
For an integrable function $f$ and $n \in \mathbb N$, consider the set: 
$$A_{f,n} = \left\{ (s,\vec z) : \left| \int f_{s,\vec z} \, d\vec F - \int f \, d\lambda \right| < 1/n, \text{ and } n \subseteq  s \right\}$$
Any fine filter with the desired property must contain each such set. It suffices to show that this family of sets has the finite intersection property.

Let $f_0,\dots,f_m$ be integrable functions and let $k>0$ be arbitrary.  Note that, if $i \leq m$, $t = \{0,\dots,n-1\}$, and $\vec x \in \N$, then $S_{f_i,2^n}(\vec x) = \int (f_i)_{t,\vec x} \, d\vec F$. 
By Jessen's Theorem, there is $\vec y \in 2^\N$ such that for all $i \leq m$, $\lim_{n \to \infty} S_{f_i,2^n}(\vec y) = \int f_i \, d\lambda$.    Let $N$ be large enough such that for all $i \leq m$, $| S_{f_i,2^N}(\vec y) - \int f_i \, d\lambda | < 1/k$.  Then $(N,\vec y) \in A_{f_i,k}$ for all $i\leq m$.
\end{proof}

Suppose $\la (X_i,F_i) : i \in I \ra$ is a sequence such that each $F_i$ is a fine filter over $[X_i]^{<\omega}$.  Let us say a product of sets $\prod_i A_i$, $A_i \subseteq X_i$ is a \emph{standard rectangle} if $\sqint \chi_{A_i} \, dF_i$ exists for every $i \in I$.

\begin{proposition}\label{bakerlike}
Suppose $\la (X_i,F_i) : i \in \N \ra$ is a sequence such that each $F_i$ is a fine filter over $[X_i]^{<\omega}$.  There is a definable fine filter $H$ 
on $[\N]^{<\omega} \times \prod_i X_i$ 
such that for every standard rectangle $A = \prod_i A_i$,
$$\sqint \chi_A \, d(\vec F,H) = \prod_i \sqint \chi_{A_i} \, dF_i,$$
and $H$ is the minimal filter with this property.
\end{proposition}

\begin{proof}
Let $A^0,\dots,A^{m-1}$ be standard rectangles, $A^j = \prod_{i \in \N} A^j_i$.
Let $r^j_i = \sqint \chi_{A^j_i} \, dF_i$.  Since $0 \leq r^j_i \leq 1$ for each $i$ and $j$, the sequence of initial products $\la r^j_0\cdots r^j_n : n \in \N \ra$ is a nondecreasing sequence of nonnegative reals, so $\prod_i r^j_i$ converges.  If $\prod_i r^j_i > 0$, then clearly $\liminf_i r^j_i = 1$.

Suppose $j$ is such that $\prod_i r^j_i = 0$.  Let $\e > 0$.  There is $n$ such that $\prod_{i<n} r^j_i < \e$.  Suppose $\vec x = \la x_i : i \in \N \ra \in \prod_i X_i$.  If $\la x_i : i \geq n \ra \in \prod_{i \geq n} A^j_i$, then $(\chi_{A^j})_{n,\vec x} = \chi_{A_0 \times \dots \times A_{n-1}}$ on $X_0 \times \dots \times X_{n-1}$.  Otherwise, if $\la x_i : i \geq n \ra \notin \prod_{i \geq n} A^j_i$, $(\chi_{A^j})_{n,\vec x} = 0$ on $X_0 \times \dots \times X_{n-1}$.  
Applying Proposition \ref{prodchar}, we get that for every $\vec x$, $\int (\chi_{A^j})_{n,\vec x} \, d\vec F< \e$.  Hence, for any fine filter $H$, $\sqint \chi_{A^j} \, d(\vec F,H) = 0$.

Suppose then, without loss of generality, that $\prod_i r^j_i > 0$ for all $j < m$.  Let $n$ be large enough such that for each $j < m$ and each $i > n$, $1- r^j_i < 1/m$.  Then for each $i > n$, there is $x_i \in \bigcap_{j<m} A^j_i$.  Hence if $\vec x$ is such that $\vec x(i) = x_i$ for $i>n$, then for each $j<m$ we have $ (\chi_{A^j})_{n,\vec x} = \chi_{A_0 \times \dots \times A_{n-1}}$ on $X_0 \times \dots \times X_{n-1}$.  Now let $\e > 0$ be arbitrary.  Let $n' \geq n$ be such that for each $j<m$,
$r^j_0 \cdots r^j_{n'} - \prod_{i \in \N} r^j_i < \e$.  Then
$$-\e < \int (\chi_{A^j})_{n',\vec x} \, d\vec F - \prod_i r^j_i  < \e.$$

For a standard rectangle $A$, let $r_A$ be the infinite product of its side lengths.  The above argument shows that the collection of sets 
$$S_{A,\e} = \left\{ (n,\vec x) :  -\e< \int (\chi_{A})_{n,\vec x} - r_A < \e \right\},$$
for standard rectangles $A$ and real $\e>0$, generates a fine filter on $[\N]^{<\omega} \times \prod_i X_i$, giving the desired conclusion.
\end{proof}

Baker \cite{10.2307/2048779} showed that it is possible to define a translation-invariant ``Lebesgue measure'' on $\R^\omega$ that is not $\sigma$-finite but assigns to rectangles of the form $\prod_{i} (a_i, b_i)$ their correct volume $\prod_{i} (b_i-a_i)$, provided that this infinite product converges .  We note that the restriction of Baker's measure to the cube $[0,1]^\omega$ gives a probability space.  The reason this does not contradict the ``no-go'' result mentioned at the beginning of \S\ref{geomsec} is that is not a normed real vector space---the ``Euclidean'' norm returns $\infty$ on many vectors therein.  The Banach spaces $\ell^p$ are contained in $\R^\omega$, but Baker's measure gives any open ball in $\ell^p$ measure zero.

The measure on $\prod_i X_i$ induced by the filter integral of Proposition \ref{bakerlike} generalizes this result to arbitrary products. Moreover, it assigns the correct measure not only to products of intervals, but to products of arbitrary measurable sets.  If each $F_i$ is the filter corresponding to a translation-invariant measure on $X_i$, then the measure of Proposition \ref{bakerlike} is also translation-invariant on the algebra generated by standard rectangles:

\begin{proposition}
Suppose $\la (X_i,\mu_i,G_i) : i \in \N \ra$ is a sequence such that for each $i$, $\mu_i$ is a finitely additive real-valued probability measure on $X_i$ without point masses, and $G_i$ is a group of $\mu_i$-invariant transformations of $X_i$.  Let $G = \prod_i G_i$, and let $G$ act on $\prod_i X_i$ by $\vec g(\vec x) = \la g_0(x_0),g_1(x_1),\dots \ra$.

Let $F_i = F_{\mu_i}$, given according to Theorem \ref{probext}, and let $H$ be the filter given by Proposition \ref{bakerlike} with respect to the sequence $\la F_i : i \in \N\ra$. Let $\mathcal A$ be the algebra of subsets of $\prod_i X_i$ generated by standard rectangles.  Then for each $A \in\mathcal A$ and $g \in G$,
\begin{enumerate}
\item $\sqint \chi_A \, d(\vec F,H)$ exists;
\item $\sqint \chi_A \, d(\vec F,H) =  \sqint \chi_{g[A]} \, d(\vec F,H)$.
\end{enumerate}
\end{proposition}

\begin{proof}
Let us say that a set $A \subseteq \prod_i X_i$ is \emph{$G$-invariant} if $\sqint \chi_A \, d(\vec F,H)$ exists and  $\sqint \chi_A \, d(\vec F,H) =  \sqint \chi_{g[A]} \, d(\vec F,H)$ for each $g \in G$.

\begin{claim}Standard rectangles are $G$-invariant.\end{claim}
\begin{proof}
Suppose $A = \prod_i A_i$ is a standard rectangle.  Then $\sqint \chi_{A_i} \, dF_i$ exists for each $i$, and by Proposition \ref{probext_measurable}, this means that $A_i$ is $\mu_i$-measurable, and thus its $G$-images have the same measure.  Thus for any $g = \la g_0,g_1,\dots\ra \in G$, $g[A]$ is a standard rectangle, and 
 $$\sqint \chi_A \, d(\vec F,H) =  \prod_i \mu_i(A_i) = \prod_i \mu_i(g_i[A_i]) = \sqint \chi_{g[A]} \, d(\vec F,H). \qedhere$$
 \end{proof}
 
 \begin{claim}Suppose $A_1,\dots,A_n$ are pairwise disjoint.  If each $A_i$ is $G$-invariant, then so is $A_1 \cup\dots\cup A_n$.  If each $A_i$ is $G$-invariant for $i<n$, and $A_1 \cup\dots\cup A_n$ is $G$-invariant, then so is $A_n$. \end{claim}
 \begin{proof}
  This follows from the fact that for any bijection $g$, $\int \chi_{g[A_1 \cup \dots \cup A_n]} \, d(\vec F,H) = \int \chi_{g[A_1]} \,d(\vec F,H) + \dots +\int \chi_{g[A_n]} \,d(\vec F,H)$.
 \end{proof}
 
 \begin{claim}Boolean combinations of standard rectangles (finite intersections of standard rectangles or their complements) are $G$-invariant.\end{claim}
 \begin{proof}
Let $A_1,\dots,A_n$ be standard rectangles.  Then $A_1 \cap\dots\cap A_n$ is also a standard rectangle and is thus $G$-invariant.  It follows from the previous claim that for any $j$, $1 \leq j \leq n$, 
$$(A_1 \cap \dots \cap A_{j-1} \cap A_{j+1} \cap \dots\cap A_n ) \setminus (A_1 \cap\dots\cap A_n)$$
 is $G$-invariant.  Thus each Boolean combination of the $A_i$'s where at most one set is complemented is $G$-invariant.   Suppose inductively that all Boolean combinations of the $A_i$'s where at most $k-1$ sets are complemented yields a $G$-invariant set.  Consider a combination in which $k$ sets are complemented.  For ease of notation assume it is 
$$A_1 \cap\dots\cap A_{n-k} \setminus (A_{n-k+1} \cup\dots\cup A_n).$$
This can be written as $A_1 \cap\dots\cap A_{n-k}$ minus the disjoint union of all Boolean combinations of the $A_i$'s in which $A_1,\dots,A_{n-k}$ appear positively, and at least one other $A_i$ appears positively.  Thus by the previous claim, the desired Boolean combination with $k$ complementations is also $G$-variant.  At the end of the induction, we get that all Boolean combinations are $G$-invariant.
 \end{proof}
 
To finish, take any set in the algebra generated by standard rectangles.  It can be written in \emph{disjunctive normal form} as a disjoint union of Boolean combinations of standard rectangles.  Thus by the above claims, it too is $G$-invariant.
\end{proof}

\bibliographystyle{amsplain.bst}
\bibliography{ultra.bib}

\end{document}